\documentclass[10pt]{article} % use larger type; default would be 10pt

\usepackage[utf8]{inputenc} % set input encoding (not needed with XeLaTeX)
  
\usepackage{tikz} 
\usetikzlibrary{arrows,automata}
\usepackage{pgfplots} 
\usepackage[hidelinks]{hyperref}

\usepackage{geometry} % to change the page dimensions
\usepackage{graphicx} % support the \includegraphics command and options
    
\usepackage{booktabs} % for much better looking tables
\usepackage{array} % for better arrays (eg matrices) in maths
\usepackage{paralist} % very flexible & customisable lists (eg. enumerate/itemize, etc.)
\usepackage{amsmath}
\usepackage{amssymb}
\usepackage{amsthm}
\usepackage{xifthen}
\usepackage{xcolor}
\usepackage{ulem}
\usepackage{hyperref}

\usepackage{mathtools}

\usepackage{subfigure}

\usepackage{soul}

\usepackage[commentmarkup=footnote]{changes}
\definechangesauthor[name={Andi}, color=red]{and}
\definechangesauthor[name={Kristoffer}, color=blue]{kri}
\definechangesauthor[name={Soeren}, color=purple]{soe}

\usepackage{fancyhdr} % This should be set AFTER setting up the page geometry
\usepackage{sectsty}
\allsectionsfont{\sffamily\mdseries\upshape} % (See the fntguide.pdf for font help)
\newtheorem{theorem}{Theorem}
\newtheorem*{acknowledgement*}{Acknowledgement}

\newtheorem{assumption}{Assumption}

\newtheorem{corollary}[theorem]{Corollary}

\newtheorem{definition}[theorem]{Definition}
\newtheorem{example}{Example}

\newtheorem{lemma}[theorem]{Lemma}

\newtheorem{problem}[theorem]{Problem}
\newtheorem{construction}[theorem]{Construction}

\newtheorem{proposition}[theorem]{Proposition}
\newtheorem{remark}[theorem]{Remark}

\usepackage[nottoc,notlof,notlot]{tocbibind} % Put the bibliography in the ToC
\usepackage[titles,subfigure]{tocloft} % Alter the style of the Table of Contents

\DeclareMathOperator{\E}{\mathbb E}

\newcommand{\kom}[1]{}
\renewcommand{\kom}[1]{{\bf [#1]}}

\usepackage{accents}
\newcommand{\ubar}[1]{\underaccent{\bar}{#1}}
\newcounter{komcounter}
\numberwithin{komcounter}{section}

\title{General Markovian randomized equilibrium existence and construction in zero-sum Dynkin games for diffusions}

\author{
Sören Christensen\\Department of Mathematics, Kiel University\\
\\Kristoffer Lindensjö\\Department of Mathematics, Stockholm University\\
}
\begin{document}
 
\maketitle

\begin{abstract} 
One of the most classical games for stochastic processes is the zero-sum Dynkin (stopping) game. We present a complete equilibrium solution to a general formulation of this game with an underlying one-dimensional diffusion. 
A key result is the construction of a characterizable global $\epsilon$-Nash equilibrium in Markovian randomized stopping times for every $\epsilon > 0$. 
This is achieved by leveraging the well-known equilibrium structure under a restrictive ordering condition on the payoff functions, leading to a novel approach based on an appropriate notion of randomization that allows for solving the general game without any ordering condition. Additionally, we provide conditions for the existence of pure and randomized Nash equilibria (with $\epsilon=0$). Our results enable  explicit identification of equilibrium stopping times and their corresponding values in many cases, illustrated by several examples.
\end{abstract}
 
\section{Introduction}\label{sec:intro}
We consider a general formulation of the classical zero-sum Dynkin game. In this setup, Player $1$ (the maximizer) selects a stopping time $\tau_1$ while Player $2$ (the minimizer) selects a stopping time $\tau_2$.  A distinctive feature of our framework is the allowance for randomized stopping times, a concept that will later play a crucial role in our analysis.
The reward of the players, which Player $1$ wants to maximize and Player $2$ wants to minimize, is 
\begin{align}\label{cost-function}
J(x;\tau_1,\tau_2):= 
\E_x\left(e^{-r(\tau_1 \wedge \tau_2)}\left(
f(X_{\tau_1})\mathbb{I}{\{\tau_1<\tau_2\}} + g(X_{\tau_2})\mathbb{I}{\{\tau_1>\tau_2\}} + h(X_{\tau_1})\mathbb{I}{\{\tau_1=\tau_2\}}
\right)\right),
\end{align}
where
$(X_t)_{t \geq 0}$ with $X_0=x$ is a regular one-dimensional linear diffusion whose state space is an interval ${\cal I} \subseteq  \mathbb{R}$, 
the functions $f,g,h:{\cal I} \rightarrow \mathbb{R}$ are continuous, and $r \geq 0$ is a constant discount rate. Further details regarding model assumptions and definitions are presented in Section \ref{sec:setup}.

An interpretation of the game is that if the maximizer stops first then the payoff for both players corresponds to $f$, 
if the minimizer stops first then the payoff for both players corresponds to $g$, and 
if both players stop at the same time then the payoff for both players corresponds to $h$. This formulation allows for various practical interpretations, such as viewing the payoff as a financial transfer from the minimizer to the maximizer. For related discussions, see, for example, \cite{baurdoux2004further, kifer2000game}.

We will now introduce three well-known, interconnected equilibrium concepts. 
 
1. For $x \in {\cal I}$, the \textit{upper value} of the game is
\begin{align*}
V^*(x):=\adjustlimits \inf_{\tau_2} \sup_{\tau_1} J(x;\tau_1,\tau_2)
\end{align*}
and the \textit{lower value} is
\begin{align*}
V_*(x):=\adjustlimits  \sup_{\tau_1} \inf_{\tau_2} J(x;\tau_1,\tau_2).
\end{align*}
The game is said to have a \textit{value} $V$ given by 
\begin{align*}
V(x):= V^*(x)= V_*(x),
\end{align*}
whenever equality of the upper and lower values holds for each $x\in  {\cal I}$, and in this case the game is said to allow a \textit{Stackelberg equilibrium}.

2. A stopping time pair $(\tau^*_1,\tau^*_2)$ is said to be an \textit{$\epsilon$-Nash equilibrium}, for a fixed constant $\epsilon \geq 0$, if
\begin{align}\label{def:SE0}
J(x;\tau_1,\tau^*_2) - \epsilon \leq J(x;\tau^*_1,\tau^*_2) \leq J(x;\tau^*_1,\tau_2)+\epsilon,
\end{align}
for all admissible stopping times $\tau_1$ and $\tau_2$, and all initial states $x \in {\cal I}$; in the vocabulary of the present paper this corresponds to the definition of a \textit{global $\epsilon$-Nash equilibrium}.

3. If a pair $(\tau^*_1,\tau^*_2)$ is a global $\epsilon$-Nash equilibrium with $\epsilon=0$, then it is said to be a \textit{global Nash equilibrium}. 

Let us discuss the relationship between the equilibrium definitions. 
It is clear that a global Nash equilibrium corresponds to a global $\epsilon$-Nash equilibrium for each $\epsilon \geq 0$. 
As expected, it holds that if a global Nash equilibrium $(\tau^*_1,\tau^*_2)$ exists then it corresponds to the unique value of the game, i.e., $J(x;\tau^*_1,\tau^*_2)=V(x)$ (see \cite[Remark 5.2]{christensen2023markovian} for an argument and references). It also holds that if a global Nash equilibrium exists for each $\epsilon>0$ then the game has a value $V$. 
Moreover, if the game has a value then an $\epsilon$-Nash equilibrium exists for each $\epsilon>0$, but only necessarily in the sense that there for each fixed $x$ exists a  pair $(\tau^*_1,\tau^*_2)$ such that \eqref{def:SE0} holds. Indeed, the existence of the value does not immediately imply the existence of a global $\epsilon$-Nash equilibrium, since that requires the pair $(\tau^*_1,\tau^*_2)$ to satisfy \eqref{def:SE0} for each $x$.

Dynkin games were introduced by Dynkin in the 1960s \cite{dynkin1969game} and have since been the subject of extensive study; several references are mentioned in Section \ref{sec:compwork}. The analysis becomes particularly tractable in the well-studied case corresponding to the ordering condition 
$f(x) \leq h(x) \leq g(x)$ for all $x$. For a general continuous-time, right-continuous, strong Markov process, this case is rigorously examined in \cite{ekstrom2008optimal}, which establishes a fundamental result on the existence of equilibria. Specifically, it is shown that the game possesses a value and admits a Nash equilibrium under the assumption that the state process is also quasi-left continuous. Moreover, this Nash equilibrium is achieved through stopping times corresponding to entry times into certain sets within the state space. In this paper, we refer to such equilibria as Markovian \textit{pure} Nash equilibria (see Section \ref{sec:setup}). Notably, when the ordering condition of \cite{ekstrom2008optimal} is satisfied, there is no need to introduce notions of randomized stopping.

\subsection{Main contributions}

When the ordering condition $f(x) \leq h(x) \leq g(x)$ is removed, the analysis becomes significantly more complex. The objective of the present paper is to search for the three notions of equilibria in Markovian pure as well as \textit{Markovian randomized stopping times}  in the context of a one-dimensional diffusion without any ordering conditions for the payoff functions; see Definition \ref{def:equilibrium-defs} for details.
By Markovian randomized stopping time we mean a particular randomized stopping time which allows an interpretation in terms of a state-dependent stopping rate where the accumulated stopping rate is allowed to increase in a manner which is not necessarily absolutely continuous with respect to the Lebesgue measure; see Definition \ref{def:puremixedstoppingtimes} and Remark \ref{rem:rand-strat}. The main contributions are:

\begin{enumerate}[(i)]
 
\item Without any ordering conditions on the payoff functions $f,g$ and $h$ we establish, for any $\epsilon>0$, the existence of a characterizable global Markovian randomized $\epsilon$-Nash equilibrium; see Definition \ref{def:equilibrium-defs} and Theorem \ref{thm:construction}. We remark that the \textit{admissible} stopping times (Definition \ref{def:admiss-ST}) are not restricted to be Markovian.

\item The equilibrium value $V$ is characterized using super- and submartingale conditions building on established methods; see Proposition \ref{asso-game:characterization} and Theorem \ref{thm:construction}. The equilibrium stopping times are explicitly constructed in Theorems \ref{thm:construction} and \ref{simplified-construction}. 

\item We provide sufficient conditions in terms of $f,g$ and $h$ guaranteeing the existence of a global Markovian randomized Nash equilibrium (i.e., with $\epsilon=0$); see  Theorem \ref{simplified-construction}.  In the literature, such an equilibrium is also called \textit{Markov perfect}. We also show that global Markovian randomized Nash equilibrium existence cannot generally be established; see Example \ref{example:non-exist-NE}. 

\item The developed theory allows for explicit solutions to specific games with randomized as well as pure equilibria; see Remark \ref{rem:identifying} and the examples in e.g., Section \ref{sec:fur-ex}.

\item We show that global Markovian pure $\epsilon$-Nash equilibrium existence cannot generally be established; cf. Example \ref{example:non-exist-NE}, as well as Theorem \ref{thm:NE-conditions2}. We hence conclude that randomization is in this sense needed for the present game. 

\item In Theorem \ref{thm:NE-conditions1} we provide new sufficient conditions for global Markovian pure Nash equilibrium existence. We also show that the sufficient ordering conditions of \cite{ekstrom2008optimal} are not necessary for the existence of such an equilibrium; see Example \ref{example-no-ordering-with-pureNE}.

\end{enumerate}

\subsection{Comparison with related work}\label{sec:compwork}
Markovian randomized stopping times defined in line with the definition of the present paper have recently been studied in the context of various stopping games in 
\cite{bodnariu2022local,decamps2023war,ekstrom2017dynkin}.
For a detailed discussion of the relationship between these works and the present paper, see Remark \ref{rem:rand-strat}.

A broader, non-Markovian formulation of the zero-sum Dynkin game is studied e.g., in \cite{laraki2005value}, where it is shown that such games admit a value in the space of general randomized stopping times analogous to our admissible stopping times. 
From this result follows that the game of the present paper has a value for each specific starting value $X_0=x$. A main advantage with the theory developed in the present paper is that we are able to identify that the corresponding $\epsilon$-Nash equilibria can---for the present diffusion setting---be found in the set of Markovian randomized stopping times, which in turn allows us to
(i) show that a global randomized $\epsilon$-Nash equilibrium exists for each $\epsilon>0$, 
(ii) construct explicit equilibrium stopping times, which have a clear game-theoretical interpretation due to the Markovian structure (cf. subgame-perfection),
(iii) characterize the equilibrium value $V=V(x)$ as the starting value $X_0=x$ varies, 
(iv) establish conditions for the existence of pure and randomized global Nash equilibria ($\epsilon=0$), and 
(v) solve specific problems explicitly. 

These types of results correspond to an approach for the treatment of optimal stopping problems. There, one typically treats Markov state processes separately from non-Markov state processes, see, for example, \cite{Peskir,shiryaev2007optimal}. 
There are two reasons for this: On the technical side, there are the characterization possibilities for solutions that are available for Markov processes, which in turn offer methods for calculating solutions, e.g., via the relationship between diffusion processes and the associated differential operators. Secondly, the Markovian nature of the optimal stopping times (first entry times) plays a crucial role in ensuring that rational agents adhere to these strategies over time. This structure aligns the stopping rules with the dynamic consistency expected in rational decision-making.

This view has until recently not been systematically adopted in the treatment of continuous time Markov process stopping games allowing for randomized stopping, and the present paper addresses this gap for the zero-sum game, making it, to the best of our knowledge, the first study in this direction. The recent study \cite{christensen2024existence} explores a related class of non-zero-sum games.
 More precisely, a class of Dynkin games of war-of-attrition type for a general one-dimensional diffusion is considered.
 Assuming an ordering condition corresponding to the players preferring to stop last and being indifferent between stopping first or simultaneously, the existence of randomized Markov-perfect equilibria ($\epsilon=0$) is established using Kakutani's theorem. The ordering condition was essential to ensuring the continuity of the underlying functionals. Although the methodology in \cite{christensen2024existence} provides a blueprint for a proof of existence of equilibria in more general settings, it differs fundamentally from our approach, which explicitly constructs $\epsilon$-Nash equilibria in Markovian randomized stopping times.  
It seems that the two methods developed are not suitable for the respective other problem and they should hence be seen as complementary to each other: The method from \cite{christensen2024existence} provides the existence of equilibria under sufficient continuity, while the method presented in this paper can be used to generalize the existing equilibria in a smaller class of games to ($\epsilon$-)equilibria in a larger class. Additionally, the present paper provides an explicit equilibrium construction allowing for equilibrium characterization.

A recent paper studying essentially the same game setup as \cite{christensen2024existence} is \cite{decamps2024mixed}. This paper also deals with the theoretical question of existence without explicit constructions, but sets up the fixed-point theorems in a slightly different way. As with \cite{christensen2024existence}, both the game formulation and the methods in \cite{decamps2024mixed} differ fundamentally from those of the present paper.

As already mentioned, the zero-sum game without randomized stopping under the additional ordering $f\leq h\leq g$ is treated in \cite{ekstrom2008optimal,peskir2009optimal}; see Section \ref{sec:assoc-game} for details. Moreover, a systematic investigation of discrete time Markovian Dynkin games when allowing for Markovian randomized stopping times---suitably defined for discrete time problems as in \cite{christensen2020Timemyopic,ferenstein2007randomized}---is presented in \cite{christensen2023markovian}.

 A complete survey of the vast literature on stopping games is not within the scope of the present paper. 
 Let us however highlight a few additional relevant contributions (see also Remark \ref{rem:rand-strat}, where we relate the notion of Markovian randomized stopping times to the literature): 
\begin{itemize}
    \item \cite{ekstrom2017dynkin} studies a two-player zero-sum stopping game under heterogeneous beliefs about the drift of the state process. Both pure and randomized Markovian equilibria are analyzed.
\item \cite{decamps2023war} examines a two-player war-of-attrition game in an SDE setting, providing an equilibrium characterization via variational inequalities.
\item \cite{de2018nash} analyzes a two-player non-zero-sum stopping game and establishes sufficient conditions for the existence of equilibria using threshold stopping strategies.
\item \cite{grun2013dynkin} uses viscosity solutions to study non-zero-sum stopping games in SDE settings, focusing on randomized stopping times.
\item \cite{martyr2021nonzero} explores connections between Markovian non-zero-sum stopping games and generalized Nash equilibrium problems.
\item For non-Markovian settings, contributions include \cite{bayraktar2017robust,de2020value,riedel2017subgame,rosenberg2001stopping,touzi2002continuous}. Connections to backward stochastic differential equations with reflection are studied in \cite{cvitanic1996backward}.
\end{itemize}

 \subsection{Structure of the paper}
The remainder of the present paper is organized as follows: 
Section \ref{sec:setup} specifies the mathematical model, provides formal definitions, and introduces an important associated game that is crucial for the subsequent analysis. 
Section \ref{sec:NE-exist} contains the general global Markovian randomized $\epsilon$-Nash equilibrium existence result and the characterization result for the value of the game.  
Section \ref{sec:howtofind} provides sufficient conditions for global Markovian randomized Nash equilibrium existence ($\epsilon=0$).  
Section \ref{sec:NE-condi} further studies the question of when different kinds of equilibria exist. Examples are provided throughout the paper. While the main ideas of the present paper are quite easy to understand, the general equilibrium construction as well as several of the proofs are technical and lengthy. To increase readability we have hence decided to relegate the general construction to Appendix \ref{appendix-construction}, and all proofs to 
Appendices \ref{proof-app-B}, \ref{proof-app-C}, \ref{proof-app-D} and \ref{proof-app-E}.

\section{Model specification, equilibrium definitions and an associated game}\label{sec:setup}
The process $X=(X_t)_{t \geq 0}$ is a one-dimensional diffusion that lives on an interval ${\cal I}=(\alpha,\beta) \subseteq  \mathbb{R}$, with $-\infty \leq \alpha < \beta \leq \infty$, and solves the stochastic differential equation (SDE)
\begin{align*}
dX_t = \mu(X_t)dt + \sigma(X_t)dW_t, \enskip X_0=x,
\end{align*}
where $\mu:{\cal I} \rightarrow \mathbb{R}$ and $\sigma:{\cal I} \rightarrow (0,\infty)$ are continuous, and
$(W_t)_{t \geq 0}$ is a standard one-dimensional Wiener process, living on a complete filtered probability space $(\Omega, {\cal F}, \mathbb{P}, ({\cal F}_t)_{t \geq 0})$ that satisfies the usual conditions. In particular, we assume that the SDE has a weak solution that is unique in law; cf. e.g., \cite[Sec. II.6.]{borodin2012handbook} or \cite[Ch. 5.5]{Karatzas2}. The corresponding probability measures and expected values are denoted by $\mathbb{P}_x(\cdot):= \mathbb{P}_x(\cdot| X_0=x)$ and $\E_x(\cdot):= \E_x(\cdot| X_0=x)$, respectively. Further model assumptions are found in Section \ref{sec:assum}.

We now define the sets of admissible stopping times in accordance with a general framework for randomized stopping times, see, e.g., \cite{ekstrom2017dynkin, ekstrom2022detect} as well as \cite{rosenberg2001stopping, touzi2002continuous}. Notably, these sets encompass any stopping time with respect to ${\cal F}^X$, which denotes the smallest filtration to which $X$ is adapted, satisfying the usual conditions. Furthermore, each randomized stopping time that allows for distinct randomization mechanisms for each player is deemed admissible. More precisely:
 
\begin{definition}[Admissible stopping times] \label{def:admiss-ST} 
A stopping time for Player $i=1,2$ is said to be admissible if it is on the form
\begin{align}\label{def:admissST-eq}
\tau_i = \tau_{\Psi^{(i)}} = \inf \{t \geq 0: \Psi_t^{(i)} \geq E_i \}
\end{align}
where (i)
$\Psi^{(i)} = (\Psi_t^{(i)})_{t \geq 0}$,  with $\Psi_{0-}^{(i)}= 0$, is a 
non-decreasing, right-continuous and ${\cal F}^X$-adapted process that takes values in $[0,\infty]$, and 
(ii) $E_i$ is an $Exp(1)$-distributed random variable (assumed to live on our probability space) that is independent of all other random sources.  

The spaces of admissible stopping times are denoted by $\mathbb{T}_i,i=1,2$.
\end{definition}
We use the convention that $\inf \emptyset = \infty$, so that for example $\Psi_t^{(i)}=0$ for all $t\geq0$ implies that $\tau_{\Psi^{(i)}}=\infty$.

We are now ready to define the notions of Markovian randomized and pure stopping times, as well as the notions of equilibria (in line with Section \ref{sec:intro}).

\begin{definition}[Markovian randomized and pure stopping times] \label{def:puremixedstoppingtimes}
The stopping time in \eqref{def:admissST-eq} is said to be a Markovian randomized stopping time if $\Psi^{(i)}$ is on the form
\begin{align}\label{eq:psi-strat}
\Psi^{(i)}_t= A^{(i)}_t  +  \infty \mathbb{I}\{\tau^{D_i} \leq t \}, \enskip t  \geq 0, 
\end{align} 
where $(A_t^{(i)})_{t \geq 0}$ is a continuous additive functional of $X$ (see e.g., \cite[Ch. X]{revuz2013continuous}) and
\begin{align*}
\tau^{D_i}:= \inf \{t \geq 0:X_t \in  D_i\},
\end{align*}
where $D_i \subseteq {\cal I}$ is a measurable (stopping) set. 

The stopping time in \eqref{def:admissST-eq} is said to be a Markovian pure stopping time if it is on the form of a first entry time, i.e., in case $A^{(i)}_t = 0,t \geq 0$ so that $\tau_{\Psi^{(i)}} = \tau^{D_i}$.
\end{definition}

\begin{definition}[Equilibrium definitions] \label{def:equilibrium-defs}\enskip 

\begin{itemize}

\item A Markovian randomized stopping time pair $(\tau^*_1,\tau^*_2)=(\tau_{\Psi^{(1)}},\tau_{\Psi^{(2)}})$ is said to be a global Markovian randomized $\epsilon$-Nash equilibrium, for a fixed constant $\epsilon\geq 0$, if
\begin{align}\label{def:SE}
\begin{split}
& J(x;\tau_1,\tau^*_2) - \epsilon \leq J(x;\tau^*_1,\tau^*_2) \leq J(x;\tau_1^*,\tau_2) + \epsilon\\
& \enskip \text{for all $\tau_1 \in \mathbb{T}_1, \tau_2 \in \mathbb{T}_1$, and $x\in {\cal I}$.}
\end{split}
\end{align}
If \eqref{def:SE} is satisfied with $\epsilon=0$ then $(\tau^*_1,\tau^*_2)=(\tau_{\Psi^{(1)}},\tau_{\Psi^{(2)}})$ is said to be a global Markovian randomized Nash equilibrium.

\item A Markovian pure stopping time pair $(\tau^*_1,\tau^*_2)=(\tau^{D_1},\tau^{D_2})$ is said to be a global Markovian pure $\epsilon$-Nash equilibrium if it satisfies \eqref{def:SE}. 
If it satisfies \eqref{def:SE} with $\epsilon=0$ then it is said to be a global Markovian pure Nash equilibrium.

\item The value of the game $V$ is defined as $V(x):=V^*(x)=V_*(x)$, whenever this equality holds for each $x\in  {\cal I}$, where 
\begin{align}\label{214fadfr32}
\begin{split}
V^*(x)&:=\adjustlimits \inf_{\tau_2 \in \mathbb{T}_2} \sup_{\tau_1 \in \mathbb{T}_1} J(x;\tau_1,\tau_2),\\
V_*(x)&:=\adjustlimits  \sup_{\tau_1\in \mathbb{T}_1} \inf_{\tau_2\in \mathbb{T}_2} J(x;\tau_1,\tau_2).
\end{split}
\end{align}

\end{itemize}

\end{definition}

\begin{remark}[Interpretation of Markovian randomized stopping]\label{rem:rand-strat} 
The interpretation of a Markovian pure stopping time is that at each time $t$ the decision to stop or not depends only on the state of $X_t$; in particular, when $A_t^{(i)}=0, t \geq 0$ we have that \eqref{eq:psi-strat} corresponds to immediate stopping on $D_i$ and not stopping on $D_i^c$. 

In the randomized case, the interpretation relies on the form of the continuous additive functionals
$(A_t)$, which, for a regular linear diffusion, can be represented as
\begin{align}\label{asfaf2}
A_t =  \int_{\cal I}\frac{l_t^y}{\sigma^2(y)}k(dy)
\end{align}
for every $t$ a.s., for some measure $k$, where $(l_t^{y})$ is the local time of $X$ at $y \in {\cal I}$, see \cite{borodin2012handbook}. This shows that our definition of Markovian randomized stopping times is in line with that of \cite[Sec. 5]{ekstrom2017dynkin} and  \cite{decamps2023war}. 

If we assume that the measure $k$ can be specified as a linear combination of a Lebesgue density $\lambda: {\cal I}\rightarrow [0,\infty)$ and Dirac measures with mass $\sigma^2(x_i)d_i>0$ on separated points $x_i \in {\cal I}$---i.e., $k(dx)= \lambda(x)dx + \sum_i d_i\sigma^2(x_i) \delta_{x_i}(dx)$--- then we find, using the occupation time formula, that \eqref{asfaf2} can be represented as 
\begin{align}\label{asdar3q432tgad}
A_t =  \int_0^t\lambda(X_s)ds + \sum_i  d_i l_t^{x_i}.
\end{align}
Indeed, the equilibrium stopping time intensities $(A_t^{(i)})$ in \eqref{eq:psi-strat} will in most examples be of the form \eqref{asdar3q432tgad}; see e.g., Section \ref{sec:fur-ex}. (Of course, a more general situation would allow the decomposition of $k$ in terms of also a singular continuous measure.) The specification \eqref{asdar3q432tgad} highlights the interpretation of a Markovian randomized stopping time in terms of a state dependent stopping rate, where $\lambda(\cdot)$ corresponds to a state-dependent stopping rate of Lebesgue density type---equivalent to stopping the first time that a Cox process jumps (\cite[Remark 2.10]{christensen2020time})---and $x_i$ correspond to points where the accumulated state dependent stopping rate increases in a singular way. 

Notably, the specification in \eqref{asdar3q432tgad} is consistent with the local time pushed mixed stopping times defined in \cite{bodnariu2022local}. For a more general justification of the term "Markovian", we refer to \cite[Sec. 3]{decamps2023war}; for a discussion of the smoothness of corresponding expectation functionals see \cite{schultz2024differentiabilityrewardfunctionalscorresponding}.
\end{remark}

\subsection{Further assumptions, notation and related observations}\label{sec:assum}

Throughout this paper, we work under the assumptions outlined in this section. The following notations will be used:
\begin{align*}
B_{1}&:= B^{g\leq h \leq f}:= \{x \in {\cal I}: g(x)\leq h(x) \leq f(x) \}\\
B_{2}&:= B^{f \lesssim h \lesssim g}:= \{x \in {\cal I}: f(x) \leq h(x) < g(x) \}  \cup \{x \in {\cal I}: f(x) < h(x) \leq g(x) \} \\
B_{3}&:= B^{h<g<f} := \{x \in {\cal I}: h(x) < g(x) < f(x) \}\\
B_{4}&:= B^{g<f<h} := \{x \in {\cal I}: g(x) < f(x) < h(x) \}\\
B_{5}&:= B^{f\leq g<h} := \{x \in {\cal I}: f(x) \leq g(x) < h(x) \}\\
B_{6}&:= B^{h<f\leq g} :=  \{x \in {\cal I}: h(x) < f(x) \leq g(x) \}.
\end{align*}
Note that e.g., $B_{1}$ and $B^{g\leq h \leq f}$ refer to the same set. While we primarily use the latter notation, the former will also prove useful in certain contexts. Additionally, observe that all sets $B_i$ are disjoint with $\cup_iB_i={\cal I}$.

\begin{assumption}[Standing assumptions] \label{assum:1}\quad

\begin{enumerate}[(i)]

\item \label{assum:A1} Each payoff function $f,g,h:{\cal I}\rightarrow \mathbb{R}$, cf. \eqref{cost-function}, can be written as the difference between convex functions. 
\item\label{assum:A2} $f,g$ and $h$ are such that points in $\cup_i  \partial(B_i)$ are separated in the sense that the number of points in $\cup_i  \partial(B_i)\cap A$ is finite for each compact $A \subseteq {\cal I}$.

\item \label{assum:A3} The boundaries of the state space ${\cal I}$ are natural.

\item \label{assum:A4} It holds that
\begin{align}\label{at-infty-assum}
\lim_{t\rightarrow  \infty} e^{-rt} f(X_t) = \lim_{t\rightarrow  \infty} e^{-rt} g(X_t) = \lim_{t\rightarrow  \infty} e^{-rt} h(X_t) = 0, \enskip \mathbb{P}_x\mbox{-a.s.}
\end{align}
We also use the convention $e^{-r\tau} w(X_\tau):=0$ on $\{\tau=\infty\}$ for any function $w: {\cal I}\rightarrow \mathbb{R}$.
Moreover, for each $x \in {\cal I}$, it holds that
\begin{align}\label{assum:integrability}
M(x):=\E_x\left( \sup_{t\geq 0}e^{-rt}|f(X_t)|  + \sup_{t\geq 0}e^{-rt}|g(X_t)| + \sup_{t\geq 0}e^{-rt}|h(X_t)|\right) < \infty.
\end{align}

\end{enumerate}
\end{assumption}

All parts of Assumption \ref{assum:1} are standard assumptions except for Assumption \ref{assum:1}\eqref{assum:A2} which can be interpreted as the payoff functions not being allowed to \textit{vary to wildly} when intersecting each other; it is moreover clear that virtually any application would satisfy this condition. Note also that this assumption implies e.g., that the sets of isolated points $\mbox{iso}(B_i),i=1,...,6$ have no accumulation points.

Let us make some useful observations indicating how the payoff functions can be decomposed into intervals where they correspond to sub- and supermartingales; which leads up to our second standing assumption (Assumption \ref{assum:2}). The observations are made for $f$, but they hold also for $g$. 

By Assumption \ref{assum:1} we have the generalized It\^o-Tanaka formula 
\begin{align*}
f(X_t) = f(x)+\int_0^t 
f'_{-}(X_s) (\mu(X_s)ds+ \sigma(X_s)dW_s)+ 
\frac{1}{2} \int_{\cal I} l_t^yf''(dy)
\end{align*}
where $f'_{-}$ is the left derivative, $f''$ is the second derivative measure, and $(l_t^y)$ is the continuous version of the symmetric local time of $X$ at $y \in \cal I$; cf. e.g., \cite[p. 68 and 75]{Peskir}. Note that the formula holds also when replacing the left derivative with the right derivative. Hence, by the occupation time formula, we also have
\begin{align*}
f(X_t) = f(x)+ \int_{\cal I}\frac{l_t^y}{\sigma^2(y)}\mathbb{K}_Xf(dy) + \int_0^tf'_{-}(X_s) \sigma(X_s)dW_s, 
\end{align*}
where the operator  $\mathbb{K}_X$ is defined so that $\mathbb{K}_Xf$ is the finite (cf. \cite[Ch. 3.6]{Karatzas2}) signed measure given by
\[\mathbb{K}_X f(dx)= \mu(x)f_{-}'(x)dx+\frac{1}{2}\sigma^2(x)f''(dx).\]
Now consider the signed measure $(\mathbb{K}_X-r)f(dy):=\mathbb{K}_X f(dy)-rf(y)dy$ and its unique Jordan decomposition 
\begin{align*}
(\mathbb{K}_X-r)f(dx)=\left( (\mathbb{K}_X-r)f \right)^{+}(dx) - \left((\mathbb{K}_X-r)f\right)^{-}(dx).
\end{align*}
We then have
\begin{align}\label{2q4124}
e^{-rt}f(X_t) = f(x)+ 
\int_0^te^{-rs}dC^{+}_s(f)-
\int_0^te^{-rs}dC^{-}_s(f)
+ \int_0^t e^{-rs}f'_{-}(X_s) \sigma(X_s)dW_s, 
\end{align}
where
\begin{align}
C_t^{+}(f):&=\int_{\cal I}\frac{l_t^y}{\sigma^2(y)}\left( (\mathbb{K}_X-r)f \right)^{+}(dy),\label{Cfplusminus1}\\
C_t^{-}(f):&=\int_{\cal I}\frac{l_t^y}{\sigma^2(y)}\left( (\mathbb{K}_X-r)f \right)^{-}(dy)\label{Cfplusminus2}.
\end{align}
We denote the support of any measure $k$ by $\mbox{\textnormal{supp}}(k)$. We are now ready to state: 

\begin{assumption}[Standing assumption]\label{assum:2} 
The points in $\partial(\mbox{\textnormal{supp}}(\left( (\mathbb{K}_X-r)f \right)^{+}))$ are separated. This holds also for 
$\partial(\mbox{\textnormal{supp}}(\left( (\mathbb{K}_X-r)f \right)^{-}))$, as well as for $\partial(\mbox{\textnormal{supp}}(\left( (\mathbb{K}_X-r)g \right)^{+}))$
and 
$\partial(\mbox{\textnormal{supp}}(\left( (\mathbb{K}_X-r)g \right)^{-}))$. 
\end{assumption} 

Let us make some clarifying comments. If $f$ is twice continuously differentiable except on a countable set of separated points $x_i \in {\cal I}, i \in I$ (where $I$ is an index set), then we have the It\^o-Tanaka formula (\cite[p. 75]{Peskir})
\begin{align*}
f(X_t) & = f(x)+
\int_0^t\left(\mu(X_s)f'(X_s)+\frac{1}{2}\sigma^2(X_s)f''(X_s)\right)\mathbb{I}_{\{X_s \neq x_i,\forall i\}}ds + \frac{1}{2} \sum_{i}  \left(f'(x_i+)-f'(x_i-)\right)l_t^{x_i} \\
& \enskip + \int_0^t \sigma(X_s)f'(X_s) \mathbb{I}_{\{X_s \neq x_i,\forall i\}}dW_s, 
 \end{align*}
and in this case $\mathbb{K}_Xf$ corresponds (cf. the occupation time formula) to 
\begin{align*}%\label{asq4234}
\mathbb{K}_X f(dx)=\left(\mu(x)f'(x)+\frac{1}{2}\sigma^2(x)f''(x)\right)\mathbb{I}\{x \neq x_i,\forall i\}dx+\frac{1}{2}\sigma^2(x_i)  \sum_{i} \left({f'(x_i+)-f'(x_i-)}\right)\delta_{x_i}(dx). 
\end{align*}
Moreover, in the twice continuously differentiable case we have
\begin{align*}
\mathbb{K}_X(dx) & = \mathbb{L}_Xf(x)dx, \enskip  \mbox{where}\\
\mathbb{L}_Xf(x): &= \mu(x)f'(x)+\frac{1}{2}\sigma^2(x)f''(x).
\end{align*}
It is immediately realized that Assumption \ref{assum:2} can be verified in most applications; essentially its interpretation corresponds to the subsets in the state space where e.g., $e^{-rt}f(X_t)$ is a strict supermartingale not being allowed to behave too wildly.

\textbf{Further notation and related observations.} For any set $B \subseteq {\cal I}$, 
$\bar{B}$ is the closure, 
$B^c$ is the complement, 
$\partial B$ is the set of boundary points, 
$\mbox{int}(B)$ is the interior 
and 
$\mbox{iso}(B)$ is the set of isolated points. 
We also use the notation $(\cdot)_+:=\max(\cdot;0)$, as well as $\tau^{B}:= \inf \{t \geq 0:X_t \in  B\}$. We rely on the following observations in the subsequent analysis. 
\begin{itemize}
\item ${B^{f \lesssim h \lesssim g}} \cup {B^{f\leq g<h}} \cup {B^{h<f\leq g}}\subseteq \{x \in {\cal I}: f(x) \leq g(x)\}=:B^{f \leq g}$.
\item ${B^{g\leq h \leq f}} \cup {B^{h<g<f}} \cup {B^{g<f<h}}\subseteq \{x \in {\cal I}: g(x) \leq f(x) \}
=: B^{g \leq f}$.
\end{itemize}
Note 
that $B^{f \leq g}$ is not generally contained in ${B^{f \lesssim h \lesssim g}} \cup {B^{f\leq g<h}} \cup {B^{h<f\leq g}}$, cf. the case when $f(x)=g(x)=h(x)$; and that $B^{g \leq f}$ is not generally subset of ${B^{g\leq h \leq f}} \cup {B^{h<g<f}} \cup {B^{g<f<h}}$, cf. the case when $f(x)=g(x)>h(x)$. However, we do have: 
\begin{itemize}
\item $B^{f < g}:=\{x \in {\cal I}: f(x) < g(x)\}\subseteq {B^{f \lesssim h \lesssim g}} \cup {B^{f\leq g<h}} \cup {B^{h<f\leq g}}$.
\item $B^{g < f}:=\{x \in {\cal I}: g(x) < f(x)\}\subseteq {B^{g\leq h \leq f}} \cup {B^{h<g<f}} \cup {B^{g<f<h}}$.
\item $\partial ({B^{f \lesssim h \lesssim g}} \cup {B^{f\leq g<h}} \cup {B^{h<f\leq g}}), \partial ({B^{g\leq h \leq f}} \cup {B^{h<g<f}} \cup {B^{g<f<h}}) \subseteq \{x \in {\cal I}: f(x)=g(x)\}=:B^{f=g}$.
\item ${B^{f \lesssim h \lesssim g}} \cup {B^{f\leq g<h}} \cup {B^{h<f\leq g}} \cup  B^{f=g}=B^{f \leq g}$.
\item ${B^{g\leq h \leq f}} \cup {B^{h<g<f}} \cup {B^{g<f<h}} \cup  B^{f=g}=B^{g \leq f}$.
\end{itemize}

\subsection{An associated game with the same value}\label{sec:assoc-game}

Our equilibrium construction relies on a connection between the original game and the following associated game, which can be viewed as a reformulation of the original game designed to satisfy the payoff function ordering $f\leq h\leq g$ assumed in \cite{ekstrom2008optimal}. As we shall see in Section \ref{sec:NE-exist}, the value of the associated game coincides with that of the original game (Definition \ref{def:equilibrium-defs}). 
To increase readability, we will use the notation $V$ for the value of both games. It is, however, important to note that the equilibrium stopping times of the associated game are generally \textit{not} equilibrium stopping times for the original game. 

Proofs for this section are found in Appendix \ref{proof-app-B}.

\begin{problem}[An associated game] \label{thm:erik-goran} The associated game is a specification of the game so far presented corresponding to the functions 
$f,g$ and $h$ being replaced with functions 
$\tilde f,\tilde g,\tilde h: {\cal I}\rightarrow \mathbb{R}$ given by:
\begin{align}\label{assoc-game-def-fgh}
\begin{split}
&\text{$\tilde f = \tilde h \leq \tilde g$ on ${\cal I}$ with:}\\ 
&\text{$\tilde f= f$ and $\tilde g=g$, on $B^{f \leq g}$, and}\\
&\text{$\tilde f = \tilde g = \tilde h$, on $B^{g \leq f}$, where 
$\tilde f = f$ on ${B^{g<f<h}}$, $\tilde g = g$ on ${B^{h<g<f}}$ and $\tilde h = h$ on ${B^{g\leq h \leq f}}$.}
\end{split}
\end{align}
\end{problem}
See Figure \ref{fig-assoc} for an illustrative example. Note that $\tilde f,\tilde g$ and $\tilde h$ are completely determined by \eqref{assoc-game-def-fgh} and that they satisfy the assumptions for $f,g$ and $h$ (Section \ref{sec:assum}).  

We remark that  having $\tilde h:=\tilde f$ is arbitrary in the sense that we could also have selected $\tilde h$ according to the condition $\tilde f  \leq \tilde h \leq \tilde g$.

The associated game is constructed in order to satisfy the standard payoff ordering assumption (compare the first part of \eqref{assoc-game-def-fgh} and  Section \ref{sec:intro}). We thus obtain the following result, the first part of which is a version of \cite[Theorem 2.1]{ekstrom2008optimal}.  
\begin{proposition}[Equilibrium existence for the associated game \cite{ekstrom2008optimal} and related observations]\label{thm:assoc-game} \enskip

\begin{enumerate}[(i)]
\item A global Markovian pure Nash equilibrium $(\tilde\tau_1^*,\tilde\tau_2^*)$, whose value corresponds to a continuous function $V: {\cal I}\rightarrow \mathbb{R}$, exists. 
More precisely,  
\begin{align*}
 \tilde \tau_i^*= \inf\{t \geq 0: X_t \in D_i^{*}\}, \enskip i=1,2,
\end{align*}
where
\begin{align}\label{def-D1-D2-sets-tilde}
\begin{split}
D_1^{*}:&=\{x \in {\cal I}: \tilde f(x)=V(x)\},\\
D_2^{*}:&=\{x \in {\cal I}: \tilde g(x)=V(x)\}.
\end{split}
\end{align}
Moreover, $\tilde f(x)\leq V(x) \leq\tilde g(x)$, for each $x \in {\mathcal I}$.

\item  For $x \in B^{g \leq f} \supseteq  {B^{g\leq h \leq f}} \cup {B^{h<g<f}} \cup {B^{g<f<h}} $, it holds that $V(x) = \tilde f(x)=\tilde g(x)$. 
Moreover, in equilibrium both players stop on $B^{g \leq f}$; in particular, 
\begin{align}\label{asdasdDd}
{B^{g\leq h \leq f}} \cup {B^{h<g<f}} \cup {B^{g<f<h}} \subseteq B^{g \leq f} \subseteq  D_i^{*} , \enskip i =1,2.
\end{align} 
\end{enumerate}
\end{proposition}

Let us now report a  characterization result for the value of the associated game. 
We remark that results similar to Proposition \ref{asso-game:characterization}(i) are, under slightly different conditions for the payoff functions, available in \cite[Theorem 3.1]{ekstrom2017dynkin} and \cite[Theorem 2.1]{peskir2009optimal}. An illustrative example is provided in Figure \ref{fig-assoc} (cf. Remark \ref{rem:eq_associated}).

\begin{figure}[ht]
	\begin{center}
  \begin{tikzpicture}
	\begin{axis}[
        xmin= -3, xmax= 3,
        ymin= -0.3, ymax= 2.2,
        xtick = {\empty}, %ytick = {0}
        clip = false]
	\addplot[thick, dotted, domain=-3:-2] {-0.1*(x+3)^2+1.1};
    \addplot[thick, dotted, domain=-2:0] {(x+1)^2} node [pos=0.4, above] {$f$};
	\addplot[thick, dotted, domain=0:1.8] {(x-1)^2};
    \addplot[thick, dotted, domain=1.8:3] {x-1.16};
        \addplot[thick, dashed, domain=-3:0] {0.2*(x+3)+0.8} node [pos=0.6, above] {$g$};
	\addplot[thick, dashed, domain=0:2] {(x-1)^2+0.4};
	\addplot[thick, dashed, domain=2:2.58] {1.4};
        \addplot[thick, dashed, domain=2.58:3] {-0.5*sqrt(x-2.58)+1.4};
	\addplot[thick, dashdotted, domain=-3:-2] {-x*0.5};
    \addplot[thick, dashdotted, domain=-2:0.2] {-x*0.5} node [pos=0.4, above] {$h$};
    \addplot[thick, dashdotted, domain=0.2:2] {-0.1};
    \addplot[thick, dashdotted, domain=2:3] {-0.1};
\addplot[color = black, densely dotted, thick] coordinates {(-2, -0.3) (-2, -0.2)};
\addplot[color = black, densely dotted, thick] coordinates {(-0.5, -0.3) (-0.5, -0.2)};
\addplot[color = black, densely dotted, thick] coordinates {(2.58, -0.3) (2.58, -0.2)};
\node [below] at (50, 0) {${B_{4}}$};
\node [below] at (105, 0) {${B_{1}}$};
\node [below] at (180, 0) {${B_{2}}$};
\node [below] at (390, 0) {${B_{6}}$};
\node [below] at (580, 0) {${B_{3}}$};
 \end{axis} 
	\end{tikzpicture}
\begin{tikzpicture}
	\begin{axis}[
        xmin= -3, xmax= 3,
        ymin= -0.3, ymax= 2.2,
        xtick = {\empty}, %ytick = {0}
        clip = false]
	\addplot[thick, dotted, domain=-3:-2] {-0.1*(x+3)^2+1.1};
    \addplot[thick, dotted, domain=-2:0] {(x+1)^2} node [pos=0.25,below] {$\tilde f = \tilde h$};
	\addplot[thick, dotted, domain=0:1.8] {(x-1)^2};
    \addplot[thick, dotted, domain=1.8:2.57] {x-1.16};
    \addplot[thick, dotted, domain=2.58:3] {-0.5*sqrt(x-2.58)+1.4};
         \addplot[thick, dashed, domain=-3:-2] {-0.1*(x+3)^2+1.1};
         \addplot[thick, dashed, domain=-2:0] {0.2*(x+3)+0.8} node [pos=0.4, above] {$\tilde g$};
	\addplot[thick, dashed, domain=0:2] {(x-1)^2+0.4};
	\addplot[thick, dashed, domain=2:2.58] {1.4};
        \addplot[thick, dashed, domain=2.58:3] {-0.5*sqrt(x-2.58)+1.4};
\addplot[color = black, densely dotted, thick] coordinates {(-2, -0.3) (-2, -0.2)};
\addplot[color = black, densely dotted, thick] coordinates {(-0.5, -0.3) (-0.5, -0.2)};
\addplot[color = black, densely dotted, thick] coordinates {(2.58, -0.3) (2.58, -0.2)};
\node [below] at (50, 0) {${B_{4}}$};
\node [below] at (105, 0) {${B_{1}}$};
\node [below] at (180, 0) {${B_{2}}$};
\node [below] at (390, 0) {${B_{6}}$};
\node [below] at (580, 0) {${B_{3}}$};
 \end{axis} 
	\end{tikzpicture}
 \begin{tikzpicture}
	\begin{axis}[
        xmin= -3, xmax= 3,
        ymin= -0.3, ymax= 2.2,
        xtick = {\empty}, %ytick = {0}
        clip = false]
	\addplot[thick, dotted, domain=-3:-2] {-0.1*(x+3)^2+1.1};
    \addplot[thick, dotted, domain=-2:0] {(x+1)^2} node [pos=0.25,below] {$\tilde f = \tilde h$};
	\addplot[thick, dotted, domain=0:1.8] {(x-1)^2};
    \addplot[thick, dotted, domain=1.8:2.57] {x-1.16};
    \addplot[thick, dotted, domain=2.58:3] {-0.5*sqrt(x-2.58)+1.4};
         \addplot[thick, dashed, domain=-3:-2] {-0.1*(x+3)^2+1.1};
         \addplot[thick, dashed, domain=-2:0] {0.2*(x+3)+0.8} node [pos=0.4, above] {$\tilde g$};
	\addplot[thick, dashed, domain=0:2] {(x-1)^2+0.4};
        \addplot[thick, dashed, domain=2:2.58] {1.4};
	\addplot[thick, dashed, domain=2.58:3] {-0.5*sqrt(x-2.58)+1.4};
    \addplot[thick, solid, domain=-3:-2] {-0.1*(x+3)^2+1.1};
    \addplot[thick, solid, domain=-2:0] {1} node [pos=0.5, below] {$V$};
    \addplot[thick, solid, domain=0:0.6] {1-0.75*x};
    \addplot[thick, solid, domain=0.6:1.31] {(x-1)^2+0.4};
    \addplot[thick, solid, domain=1.31:2.58] {0.72*(x-2.58)+1.4};
    \addplot[thick, solid, domain=2.58:3] {-0.5*sqrt(x-2.58)+1.4};
\addplot[color = black, densely dotted, thick] coordinates {(-2, -0.3) (-2, -0.2)};
\addplot[color = black, densely dotted, thick] coordinates {(-0.5, -0.3) (-0.5, -0.2)};
\addplot[color = black, densely dotted, thick] coordinates {(2.58, -0.3) (2.58, -0.2)};
\node [below] at (50, 0) {${B_{4}}$};
\node [below] at (105, 0) {${B_{1}}$};
\node [below] at (180, 0) {${B_{2}}$};
\node [below] at (390, 0) {${B_{6}}$};
\node [below] at (580, 0) {${B_{3}}$};
 \end{axis} 
	\end{tikzpicture}
   \begin{tikzpicture}
	\begin{axis}[
        xmin= -3, xmax= 3,
        ymin= -0.3, ymax= 2.2,
        xtick = {\empty}, %ytick = {0}
        clip = false]
	\addplot[thick, dotted, domain=-3:-2] {-0.1*(x+3)^2+1.1};
    \addplot[thick, dotted, domain=-2:0] {(x+1)^2} node [pos=0.4, above] {$f$};
	\addplot[thick, dotted, domain=0:1.8] {(x-1)^2};
    \addplot[thick, dotted, domain=1.8:3] {x-1.16};
        \addplot[thick, dashed, domain=-3:0] {0.2*(x+3)+0.8} node [pos=0.6, above] {$g$};
	\addplot[thick, dashed, domain=0:2] {(x-1)^2+0.4};
	\addplot[thick, dashed, domain=2:2.58] {1.4};
        \addplot[thick, dashed, domain=2.58:3] {-0.5*sqrt(x-2.58)+1.4};
    \addplot[thick, dashdotted, domain=-3:-2] {-x*0.5};
    \addplot[thick, dashdotted, domain=-2:0.2] {-x*0.5} node [pos=0.4, above] {$h$};
    \addplot[thick, dashdotted, domain=0.2:2] {-0.1};
    \addplot[thick, dashdotted, domain=2:3] {-0.1};
    \addplot[thick, solid, domain=-3:-2.02] {-0.1*(x+3)^2+1.1};
    \addplot[thick, solid, domain=-2:0] {1} node [pos=0.5, below] {$V$};
    \addplot[thick, solid, domain=0:0.6] {1-0.75*x};
    \addplot[thick, solid, domain=0.6:1.31] {(x-1)^2+0.4};
    \addplot[thick, solid, domain=1.31:2.58] {0.72*(x-2.58)+1.4};
    \addplot[thick, solid, domain=2.58:3]{-0.5*sqrt(x-2.58)+1.4};
\addplot[color = black, densely dotted, thick] coordinates {(-2, -0.3) (-2, -0.2)};
\addplot[color = black, densely dotted, thick] coordinates {(-0.5, -0.3) (-0.5, -0.2)};
\addplot[color = black, densely dotted, thick] coordinates {(2.58, -0.3) (2.58, -0.2)};
\node [below] at (50, 0) {${B_{5}}$};
\node [below] at (105, 0) {${B_{1}}$};
\node [below] at (180, 0) {${B_{2}}$};
\node [below] at (390, 0) {${B_{6}}$};
\node [below] at (580, 0) {${B_{3}}$};
 \end{axis} 
	\end{tikzpicture}
		\end{center}
			\caption{The first picture shows some payoff functions $f,g$ and $h$. The short vertical dotted lines correspond to points in $\partial B_i$ (the sets $B_i$ are defined in Section \ref{sec:assum}). The second picture shows the payoff functions $\tilde f, \tilde g$ and $\tilde h$ of the associated game. The third picture shows also the value function of the associated game $V$, assuming $X$ is a Wiener process (and $r>0$ is very small). The last picture shows the original payoff functions $f,g$ and $h$ together with the corresponding value function $V$ (for the original game), which, as we shall see, coincides with the value function of the associated game; see Theorem \ref{thm:construction} in Section \ref{sec:NE-exist}.} 
			\label{fig-assoc}
\end{figure}
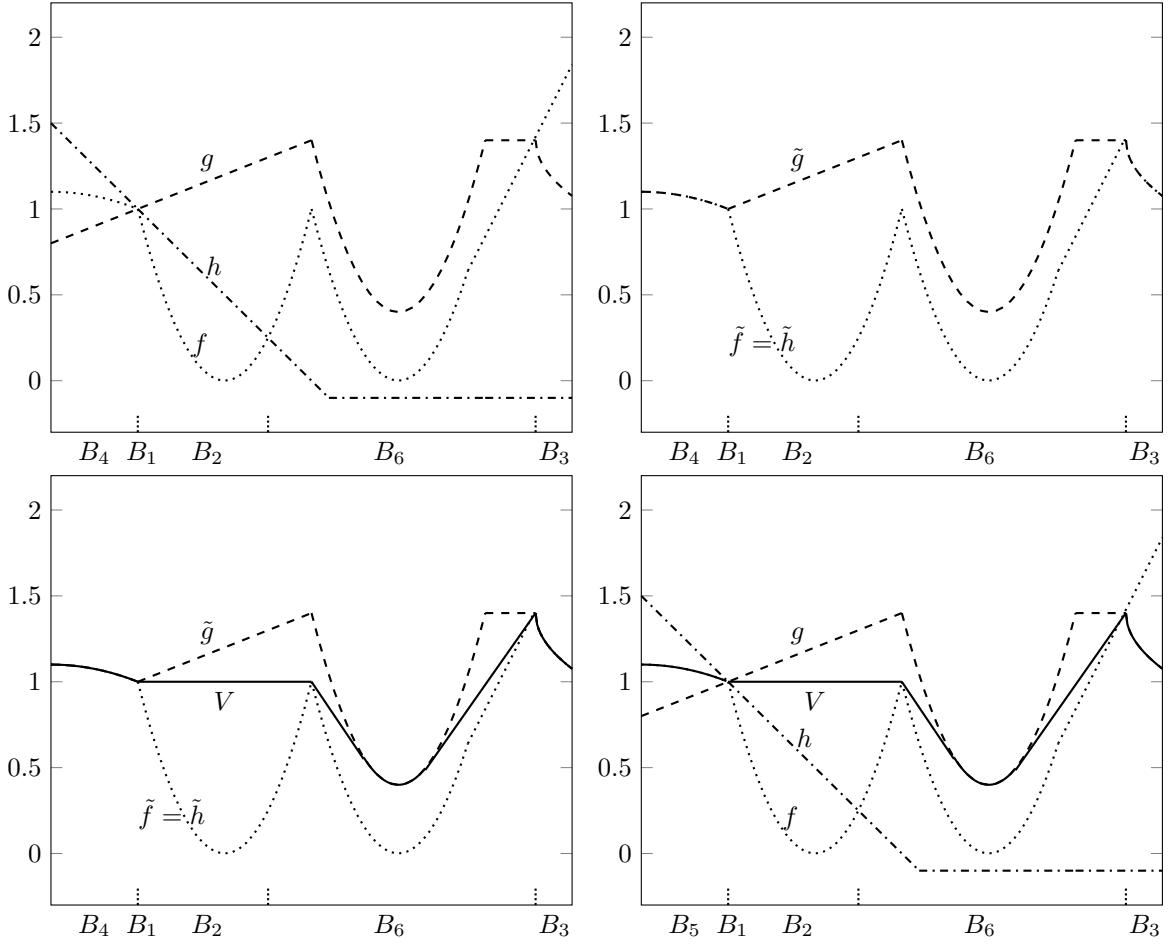

\begin{proposition}[Equilibrium characterization]\label{asso-game:characterization}\enskip 

\begin{enumerate}[(i)]
\item  Let $U:{\mathcal I}\rightarrow \mathbb{R}$ be a continuous function and define
\begin{align*}
\tau^{D_i}&:= \inf\{t \geq 0: X_t \in D_i\}, \enskip i=1,2,\\ 
D_1&:=\{x \in {\cal I}: \tilde f(x)= U(x)\},\\
D_2&:=\{x \in {\cal I}: \tilde g(x)= U(x)\}.
\end{align*}
Suppose that $e^{-r(t\wedge \tau^{D_1})}U(X_{t\wedge \tau^{D_1}}), 0 \leq t \leq \infty$ is a submartingale, that $e^{-r(t\wedge \tau^{D_2})}U(X_{t\wedge \tau^{D_2}}), 0 \leq t \leq \infty$ is a supermartingale, for each $X_0=x \in {\mathcal I}$ and that $\tilde f(x)\leq U(x)\leq\tilde g(x)$, for each $x \in {\mathcal I}$. Then, the equilibrium in Proposition \ref{thm:assoc-game} is given by $V=U$ and $D_i^*=D_i, \enskip i=1,2$. 

\item It holds that $\in  D_1^* \cap D_2^*=B^{g \leq f}$ and the value function of the associated game satisfies 
\begin{align*}%\label{123123-1}
V(x)= 
\begin{cases}
    \tilde f(x), 				&  x \in  D_1^*\\
    \tilde g(x), 				&  x \in  D_2^*\\
   \tilde f(x) =\tilde g(x), 				&  x \in  D_1^* \cap D_2^*=B^{g \leq f},
		 \end{cases}
\end{align*}
as well as
\begin{align}
\begin{split}\label{123123-2}
V(x)= 
\begin{cases}
			h(x), &  x \in {B^{g\leq h \leq f}} \subseteq D_1^* \cap D_2^*= B^{g \leq f}\\
			g(x), &  x \in {B^{h<g<f}} \subseteq D_1^* \cap D_2^*=B^{g \leq f}\\
			f(x), &  x \in {B^{g<f<h}} \subseteq D_1^* \cap D_2^*=B^{g \leq f},
		 \end{cases}
\end{split}
\end{align}
and
\begin{align}
\begin{split}\label{123123-3}
V(x)= 
\begin{cases}
			f(x), &  x \in  D_1^*\cap B^{f \leq g}\\
			g(x), &  x \in  D_2^* \cap B^{f \leq g}\\
            f(x)=g(x), &  x \in  D_1^*\cap D_2^* \cap B^{f \leq g}= B^{f = g}.
		 \end{cases}
\end{split}
\end{align}
Moreover, $f(x)\leq V(x)\leq g(x)$ or $g(x)\leq V(x)\leq f(x)$, for each $x\in {\cal I}$.
\end{enumerate}
\end{proposition}

\begin{remark}[Finding the equilibrium in the associated game] \label{rem:eq_associated}
The problem of determining the value function $V$ under the standard ordering condition $\tilde f \leq \tilde h \leq \tilde g$ has been widely studied in the literature. A particularly intuitive geometric description arises from the semiharmonic characterization: for an underlying Wiener process, the value function can be visualized as the result of "pulling a rope" between the "two obstacles" $\tilde f$ and $\tilde g$; see the lower left picture in Figure \ref{fig-assoc} for an illustration. For more general diffusion processes, the same principle applies after an appropriate transformation (for a detailed discussion see e.g., \cite{peskir2009optimal}).

Let us furthermore, in  analytical terms and under appropriate smoothness assumptions, describe a method for finding a candidate for the equilibrium value for  the associated game $V$. The method could be described in terms of the payoff functions for the associated game i.e., $\tilde f, \tilde g$ and $\tilde h$, but we will instead provide a description in terms of $f,g$ and $h$; where we rely on the relationships $\tilde f = \tilde h = f$ and $g = \tilde g$ on $B^{f\leq g}$, and $\tilde f = \tilde g = \tilde h$ on $B^{g \leq f}$ (see \eqref{assoc-game-def-fgh}):

Try to identify a sufficiently smooth function $U$ with $f\leq U\leq g$ on $B^{f<g}$, satisfying the system
\begin{align*}
\max\{(\mathbb{L}_X-r)U(x);f(x)-U(x)\}&=0, \mbox{ for $x \in B^{f< g}$ with } U(x)<g(x),\\
\min\{(\mathbb{L}_X-r)U(x);g(x)-U(x)\}&=0, \mbox{ for $x \in B^{f<g}$ with } U(x)>f(x), 
\end{align*}
with $U=f=g$ on $B^{f=g}$, which in many cases involves smooth fit conditions, e.g., $U'(x)=f'(x)$ for $x$ with $f(x)=U(x)$ and $U(x)<g(x)$. The candidate equilibrium value is then 
\begin{align*}
V(x)= 
\begin{cases}
 \tilde f(x) = \tilde g(x) = \tilde h(x) = h(x), &  x \in {B^{g\leq h \leq f}} \\
 \tilde f(x) = \tilde g(x) = \tilde h(x) = g(x), &  x \in {B^{h<g<f}}\\
 \tilde f(x) = \tilde g(x) = \tilde h(x) = f(x), &  x \in {B^{g<f<h}}\\
\tilde f(x) = \tilde g(x) = \tilde h(x) = f(x)=g(x), &  x \in {B^{f=g}}\\
 U(x), &  x \in B^{f< g}={B^{f \lesssim h \lesssim g}} \cup  {B^{f< g<h}} \cup {B^{h<f< g}}.
\end{cases}
\end{align*}
All lines except the last above corresponds to both players stopping if the payoff functions of the associated game coincide, combined with \eqref{assoc-game-def-fgh}. One verifies that the candidate $V$ is indeed an equilibrium value by verifying that
$e^{-r(t\wedge \tau^{D_1^{*}})}V(X_{t\wedge \tau^{D_1^{*}}}), 0 \leq t \leq \infty$ is a submartingale and that 
$e^{-r(t\wedge \tau^{D_2^{*}})}V(X_{t\wedge \tau^{D_2^{*}}}), 0 \leq t \leq \infty$ is a supermartingale, for each $X_0=x \in {\mathcal I}$; cf. Proposition \ref{asso-game:characterization}.

We remark that the equilibrium value $V$ is generally not differentiable even when all payoff functions are smooth (cf. Example \ref{exampleyiweh}, below). 
\end{remark}

\section{General \texorpdfstring{$\epsilon$}{e}-equilibrium existence, construction and characterization}\label{sec:NE-exist} 
We are now ready to report our first main result,  Theorem \ref{thm:construction}, which contains the general equilibrium existence and characterization results for the original game. We remark that a version of this result with more explicitly stated equilibrium stopping times, which holds under additional model assumptions, is reported in Theorem \ref{simplified-construction} in Section \ref{sec:howtofind}. An equilibrium uniqueness result is reported in the present section in Corollary \ref{Equilibrium-uniqueness}. 
Proofs for this section are found in Appendix  \ref{proof-app-C}.

Before reporting Theorem \ref{thm:construction} we present, in order to increase readability, its interpretation:   
\enskip 
\begin{itemize}

\item Recall that $D^*_i=1,2$ are the stopping sets of the associated game (Proposition \ref{thm:assoc-game}).

\item The interpretation of the equilibrium stopping times of Theorem \ref{thm:construction} (below) for the original game is that Player $1$ (the maximizer) stops without randomization whenever the state process $X$ is in the set $D_1^{*} \backslash ({B^{h<g<f}} \cup {B^{h<f\leq g}})$ and stops with 
randomization on the set $D_1^{*} \cap ({B^{h<g<f}} \cup {B^{h<f\leq g}})$, on which simultaneous stopping gives $h< f \wedge g$, using a combination of a Lebesgue density stopping rate and \textit{local time pushes} in the accumulated stopping rate; see \eqref{Ne-stretegy2} below. Local time pushes are in particular used on isolated points of $D_1^{*} \backslash ({B^{h<g<f}} \cup {B^{h<f\leq g}})$. Analogously, Player $2$ stops without randomization on 
$D_2^{*} \backslash ({B^{g<f<h}} \cup {B^{f\leq g<h}})$ and stops with 
randomization on $D_1^{*} \cap ({B^{g<f<h}} \cup {B^{f\leq g<h}})$. 
In particular, recalling \eqref{def-D1-D2-sets-tilde}, \eqref{123123-2}, \eqref{123123-3} and $ {B^{f \lesssim h \lesssim g}} \cup {B^{f\leq g<h}} \cup {B^{h<f\leq g}}\subseteq B^{f \leq g}$, we have: 
\begin{itemize}
\item On ${B^{g\leq h \leq f}} \subseteq D^*_i,i=1,2$, both players stop immediately. 

\item On ${B^{h<g<f}} \subseteq D^*_i,i=1,2$, Player $2$ stops immediately while Player $1$ randomizes. The interpretation for ${B^{g<f<h}} \subseteq D^*_i,i=1,2$ is analogous to that of ${B^{h<g<f}}$, with Player $2$ randomizing and Player $1$ stopping.

\item On ${B^{f \lesssim h \lesssim g}}$ both players stop without randomization according $D_i^*$, i.e., Player $i$ stops immediately on ${B^{f \lesssim h \lesssim g}} \cap D_i^{*}$ and does nothing on ${B^{f \lesssim h \lesssim g}} \backslash D_i^{*}$.

\item On ${B^{f\leq g<h}}$, Player $1$ stops without randomization according $D_1^*$, i.e., Player $1$ stops immediately on ${B^{f\leq g<h}} \cap D_1^{*}$ and does nothing on ${B^{f\leq g<h}} \backslash D_1^{*}$. Moreover, on ${B^{f\leq g<h}}$, Player $2$ randomizes whenever $D_2^*$ prescribes stopping, i.e., Player $2$ randomizes on ${B^{f\leq g<h}} \cap D_2^{*}$ and does nothing on ${B^{f\leq g<h}} \backslash D_2^{*}$. The interpretation for ${B^{h<f\leq g}}$ is analogous to that of ${B^{f\leq g<h}}$, with Player $1$ randomizing.

\end{itemize} 
\item It never happens that both players randomize at the same time. 
\end{itemize}

\begin{theorem}[General equilibrium existence, construction and characterization] \label{thm:construction} \enskip

\begin{enumerate}[(i)]
\item The equilibrium value of our game (Definition \ref{def:equilibrium-defs}) exists and coincides with the equilibrium value $V$ of the associated game (Proposition \ref{thm:assoc-game}). In particular, Proposition \ref{asso-game:characterization} provides a characterization of the equilibrium value and Figure \ref{fig-assoc} provides an illustrative example.

\item  A global Markovian randomized $\epsilon$-Nash equilibrium exists for each $\epsilon>0$, and the equilibrium value $V$ can be attained using Markovian randomized stopping times. 
In particular, for each fixed $\epsilon>0$, the pair $(\tau_{\Psi^{(1)}},\tau_{\Psi^{(2)}})$ is a global Markovian randomized $\epsilon$-Nash equilibrium when
\begin{align}\label{Ne-stretegy}
\begin{split}
\Psi^{(1)}_t &:= A_t^{(1)}+\infty \mathbb{I}\{\tau^{D_1^{*} \backslash ({B^{h<g<f}} \cup {B^{h<f\leq g}})} \leq t \},\\
\Psi^{(2)}_t &:= A_t^{(2)}+\infty \mathbb{I}\{\tau^{D_2^{*} \backslash ({B^{g<f<h}} \cup {B^{f\leq g<h}})} \leq t \}, 
\end{split}
\end{align}
with 
\begin{align}\label{Ne-stretegy2}
\begin{split}
A_t^{(1)}:&= \int_0^t \gamma^{(1)}_\epsilon(X_s)  \mathbb{I}{\{X_s \in  D_1^{*} \cap (B^{h<g<f} \cup {B^{h<f\leq g}})  \}}ds + \sum_{y \in \mbox{iso}(D_1^{*} \cap (B^{h<g<f} \cup {B^{h<f\leq g}}) )} \Gamma_\epsilon^{1}(y) l_t^y,\\
A_t^{(2)}:&= \int_0^t \gamma^{(2)}_\epsilon(X_s)  \mathbb{I}{\{X_s \in  D_2^{*} \cap 
({B^{g<f<h}} \cup {B^{f\leq g<h}}) \}}ds + \sum_{y \in \mbox{iso}(D_2^{*} \cap ({B^{g<f<h}} \cup {B^{f\leq g<h}}))} \Gamma_\epsilon^{2}(y) l_t^y,
\end{split}
\end{align}
where $\gamma^{(i)}_\epsilon(\cdot)$ and $\Gamma^{(i)}_\epsilon(\cdot)$ are real-valued functions $i=1,2$ (specified in Constructions \ref{construction1}--\ref{construction2}, Appendix \ref{appendix-construction}).
\end{enumerate}
\end{theorem}
The following result demonstrates that increasing the level of randomization in the equilibrium stopping times does not affect the existence of an equilibrium (provided it does not result in immediate stopping) and that the equilibrium value function is uniquely determined.

\begin{corollary}[Equilibrium uniqueness] \label{Equilibrium-uniqueness}\enskip 

\begin{enumerate}[(i)]
\item The equilibrium stopping time pair of Theorem \ref{thm:construction} is not unique. 
In particular, suppose $\Psi^{(i)}_t,i=1,2$ in \eqref{Ne-stretegy} correspond to a global Markovian randomized $\epsilon$-Nash equilibrium for some $\epsilon \geq 0$. 
Now 
replace $\gamma^{(i)}_\epsilon(\cdot)$ in \eqref{Ne-stretegy2} with a real-valued, measurable and locally bounded function that dominates $\gamma^{(i)}_\epsilon(\cdot)$, and 
replace $\Gamma^{(i)}_\epsilon(\cdot)$ in \eqref{Ne-stretegy2} with a real-valued function that dominates $\Gamma^{(i)}_\epsilon(\cdot)$. Then \eqref{Ne-stretegy} still corresponds to a global Markovian randomized $\epsilon$-Nash equilibrium for the same $\epsilon$ (although it may also correspond to a global Markovian randomized $\epsilon$-Nash equilibrium for a smaller value for $\epsilon\geq 0$). 

\item The equilibrium value of the game is (uniquely) given by $V$ in Theorem \ref{thm:construction} (and Proposition \ref{asso-game:characterization}).  

\end{enumerate}
\end{corollary}

\section{How to find equilibria, and examples}\label{sec:howtofind}

The construction of the equilibrium stopping intensities \eqref{Ne-stretegy}-\eqref{Ne-stretegy2} is technical and lengthy; see Appendix \ref{appendix-construction}.  However, under additional assumptions the construction simplifies significantly, and additionally a randomized Nash equilibrium exists. These findings are reported in our second main result, which is Theorem \ref{simplified-construction} below. The proof is found Appendix in \ref{proof-app-D}.

\begin{theorem}[Nash equilibrium existence ($\epsilon=0$)  and explicit equilibrium construction] \label{simplified-construction} 
Suppose ${B^{f\leq g<h}}\cup {B^{h<f\leq g}} = \emptyset$, i.e., if $f(x)\leq g(x)$ then $f(x) \leq h(x)\leq g(x)$ for all $x \in {\cal I}$. Then, the equilibrium $(\tau_{\Psi^{(1)}},\tau_{\Psi^{(2)}})$ of Theorem \ref{thm:construction} can be attained with $\epsilon=0$, i.e., a global Markovian randomized  Nash equilibrium exists. More precisely, the global Markovian randomized Nash equilibrium can be attained with explicit randomization according to \eqref{Ne-stretegy} and 
\begin{align}
A_t^{(1)}:=& \int_0^t \frac{1}{f(X_s)-g(X_s)}\mathbb{I}_{\{X_s\in {B^{h<g<f}}\}}dC_s^{-}(g)\label{simplified-construction:eq1}\\
A_t^{(2)}:=& \int_0^t \frac{1}{f(X_s)-g(X_s)}\mathbb{I}_{\{X_s\in {B^{g<f<h}}\}}dC_s^{+}(f)\label{simplified-construction:eq2}, 
\end{align}
where $C_t^{+}(f), t \geq 0$ is defined in \eqref{Cfplusminus1} and $C_t^{-}(g), t \geq 0$ is defined as in \eqref{Cfplusminus2} with $g$ instead of $f$ (note that both are non-decreasing processes). 

Furthermore, suppose that $f$ and $g$ are twice continuously differentiable except on countable sets of separated points, denoted by $x_i \in {\cal I}, i \in I$ and $x_j \in {\cal I}, j \in J$ (where $I$ and $J$ are index sets), respectively. Then \eqref{simplified-construction:eq1}--\eqref{simplified-construction:eq2} can be written as
\begin{align*}
A_t^{(1)}= \int_0^t\left(\frac{(\mathbb{L}_X-r)g(X_s)}{g(X_s)-f(X_s)}\right)_+\mathbb{I}_{\{X_s\in {B^{h<g<f}}, X_s \neq x_j, \forall j\}}ds + \frac{1}{2}\sum_{j}  \mathbb{I}_{\{x_j \in {B^{h<g<f}}\}}\left(\frac{g'(x_j+)-g'(x_j-)}{{g(x_j)-f(x_j)}}\right)_+l_t^{x_j}\\
A_t^{(2)}= \int_0^t\left(\frac{(\mathbb{L}_X-r)f(X_s)}{f(X_s)-g(X_s)}\right)_+\mathbb{I}_{\{X_s\in {B^{g<f<h}}, X_s \neq x_i,\forall i\}}ds + \frac{1}{2} \sum_{i}  \mathbb{I}_{\{x_i \in {B^{g<f<h}}\}}\left(\frac{f'(x_i+)-f'(x_i-)}{{f(x_i)-g(x_i)}}\right)_+l_t^{x_i}.
\end{align*}

 \end{theorem}

 \begin{remark}[Finding the equilibrium] \label{rem:identifying} 
The value of the original game is determined by the value of the associated game (see Theorem \ref{thm:construction}(i)). Hence, finding the value function $V$ reduces to the well-established framework under the standard ordering condition; in particular, Remark \ref{rem:eq_associated} provides an explicit method for identifying a candidate value function.

Moreover, the corresponding stopping sets in the original game---on which the players either stop immediately or randomize according to \eqref{Ne-stretegy}--\eqref{Ne-stretegy2} for sufficiently large functions $\gamma^{(i)}, \Gamma^{(i)}$---are   
\begin{align*}
D_1^{*}& =\{x \in {B^{f \lesssim h \lesssim g}} \cup {B^{f\leq g<h}} \cup {B^{h<f\leq g}}: f(x)= U(x)\}\cup {B^{g\leq h \leq f}} \cup {B^{h<g<f}} \cup {B^{g<f<h}}\\
D_2^{*}&=\{x \in {B^{f \lesssim h \lesssim g}} \cup {B^{f\leq g<h}} \cup {B^{h<f\leq g}}: g(x)= U(x)\}\cup {B^{g\leq h \leq f}} \cup {B^{h<g<f}} \cup {B^{g<f<h}}.
\end{align*}
If we suppose that the assumptions of Theorem \ref{simplified-construction} hold, then this result together with the observations above gives an explicit description of the equilibrium stopping times. 
In particular, 
(i) 
Player $1$ does nothing on $(D_1^*)^c$, 
randomizes according to the intensity \eqref{simplified-construction:eq2} on ${B^{h<g<f}}= D_1^* \cap {B^{h<g<f}}$, 
and stops on $D_1^* \backslash {B^{h<g<f}}$,  while 
(ii) 
Player $2$ does nothing on $(D_2^*)^c$, 
randomizes according to the intensity \eqref{simplified-construction:eq2} on ${B^{g<f<h}}= D_2^* \cap {B^{g<f<h}}$, and
stops on $D_2^* \backslash {B^{g<f<h}}$.

\end{remark}

\subsection{Examples}\label{sec:fur-ex}

Using the developed theory we here study some examples. 

\begin{example}[Randomization only in terms of a local time push]\label{example1MMM}  
Let $X$ be a Wiener process, $r>0$ and
\begin{align*}
h(x)> f(x)=|x|+1 >g(x), \enskip & x \in (-1,1)\\
h(x)= f(x)= g(x), \enskip &x \in \mathbb{R}\backslash (-1,1),
\end{align*}
where these functions are further specified to fulfill our assumptions (Section \ref{sec:setup}). Note that ${B^{g<f<h}}=(-1,1)$ and ${B^{g\leq h \leq f}}=\mathbb{R}\backslash (-1,1)$. Using the developed theory as described in Remark \ref{rem:identifying}, in particular relying on Theorem \ref{simplified-construction}, we conclude that 
$D_1^*= D_2^*=\mathbb{R}$ and find a global Markovian randomized Nash equilibrium with
$V(x)=f(x)$, that is attained by 
$\tau_{\Psi^{(1)}} = \inf\{t\geq 0: X_t \in \mathbb{R}\}$ and 
$\tau_{\Psi^{(2)}}$ with 
\begin{align}\label{q341fsf}
\begin{split}
\Psi^{(2)}_t 
& = \int_0^t\left(\frac{(\mathbb{L}_X-r)f(X_s)}{f(X_s)-g(X_s)}\right)_+\mathbb{I}\{X_s \in (-1,1)\backslash\{0\}\}ds + \frac{1}{2} \left(\frac{f'(0+)-f'(0-)}{{f(0)-g(0)}}\right)_+l_t^{0}
+\infty \mathbb{I}\{\tau^{ \mathbb{R}\backslash (-1,1)  } \leq t \}\\
& = \frac{l_t^{0}}{{f(0)-g(0)}} + \infty \mathbb{I}\{\tau^{ \mathbb{R}\backslash (-1,1)  } \leq t \}.
\end{split}
\end{align}
An interpretation of \eqref{q341fsf} is that there is a local time push in the accumulated stopping intensity of Player $2$ when the state processes visits $0$, and that Player $2$ stops without randomization on $ \mathbb{R}\backslash (-1,1)$. 

Note, however, that Player $1$ will in equilibrium stop immediately for any initial state, and that Player $2$ will hence in equilibrium never have the chance stop at $0$ (cf. Proposition \ref{prop:simul-stopping2X} in Appendix \ref{proof-app-D}). An interpretation of Player $2$'s equilibrium stopping time, corresponding to \eqref{q341fsf}, is that it ensures that Player $1$ has no incentive to deviate from the equilibrium strategy: If Player 2 did not randomize with an intensity at least as large as in \eqref{q341fsf}, Player $1$ would observe a strict submartingale around 0 and would not stop. In this sense, Player $2$ uses randomization corresponding to \eqref{q341fsf} in order to effectively threaten Player $1$ with the risk of obtaining $g(0)$ (instead of $f(0)$) if Player $1$ would allow the state process to visit $0$. 
\end{example}

\begin{example}[Randomization in terms of local time push and Lebesgue density]\label{examplebhkasdyuiqw} Let $X$ be a Wiener process, $r>0$ and
\begin{align*}
h(x)>f(x)= x^2\mathbb{I}\{x\geq 2\}+ 2x\mathbb{I}\{x<2 \}>g(x), \enskip x \in \mathbb{R}. 
\end{align*}
Note that ${B^{g<f<h}}=\mathbb{R}$. Using arguments similar to those in the example above, we find
$D_1^*= D_2^*=\mathbb{R}$ and a global Markovian randomized Nash equilibrium with 
$V(x)=f(x)$ that is attained by 
$\tau_{\Psi^{(1)}} = \inf\{t\geq 0: X_t \in \mathbb{R}\}=0$ and 
$\tau_{\Psi^{(2)}}$ with
\begin{align*}
\Psi^{(2)}_t 
& = \int_0^t\lambda_2(X_s)ds + c_2l_t^{2}
\end{align*}
where $c_2 = \frac{1}{2} \left(\frac{f'(2+)-f'(2-)}{{f(2)-g(2)}}\right)_+ = \frac{1}{4-g(2)}$ and 
\begin{align*}
\lambda_2(x) 
= \left(\frac{(\mathbb{L}_X-r)f(x)}{f(x)-g(x)}\right)_+\mathbb{I}\{x\neq 2\}
= \begin{cases}
			\left(\frac{-2rx}{2x-g(x)}\right)_+, &  x  < 2\\
			0, & x= 2\\
			\left(\frac{1-rx^2}{x^2-g(x)}\right)_+, & x >2\\
		 \end{cases}
= \begin{cases}
			\frac{-2rx}{2x-g(x)}, &  x  \leq 0\\
			0, &  0 <x  \leq 2\\
			\frac{1-rx^2}{x^2-g(x)}, & 2 < x < \frac{1}{\sqrt{r}}\\
			0, & x \geq  \frac{1}{\sqrt{r}}
		 \end{cases}.
\end{align*}
See Figure \ref{fig:examplebhkasdyuiqw} for an illustration. 
\end{example}

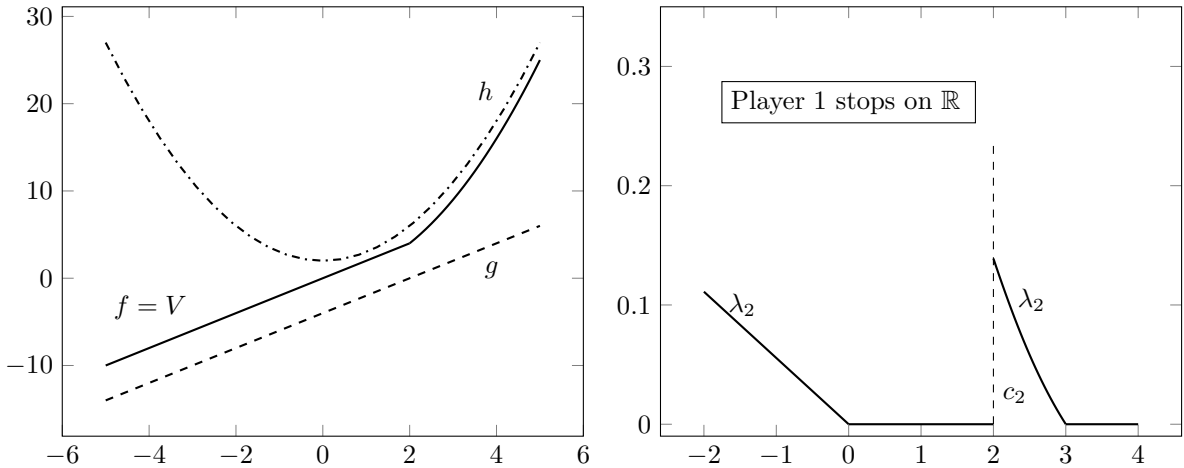
\begin{figure}[ht]
	\begin{center}
		\begin{tikzpicture}
	\begin{axis}%[ymin= -0.5, ymax= 2.5]
	\addplot[thick, solid, domain=-5:2] {2*x} node [pos=0.3, above left] {$f=V$};
	\addplot[thick, solid, domain=2:5] {x^2};
	\addplot[thick, dashed, domain=-5:5] {2*x-4} node [pos=0.85, below right] {$g$};
	\addplot[thick, dashdotted, domain=-5:5] {x^2+2} node [pos=0.85, above left] {$h$};
	\end{axis} 
	\end{tikzpicture}
			\begin{tikzpicture}
	\begin{axis}[ymin= -0.01, ymax= 0.35]
	\addplot[thick, solid, domain=-2:0] {(-2*x/9)/(2*x-(2*x-4))} node [pos=0.1, right] {$\lambda_2$};
	\addplot[thick, solid, domain=0:2] {0};
	\addplot[thick, solid, domain=2:3] {(1-x^2/9)/(x^2-(2*x-4))}  node [pos=0.2, right] {$\lambda_2$};
	\addplot[thick, solid, domain=3:4] {0};
	%
	%\draw (2,2) circle (3cm);
	\draw[dashed, dashed] (400,10) -- (400,250) node [pos=0.1, right] {$c_2$};
		\node[draw,text width=3.1cm] at (200,280) {Player $1$ stops on $\mathbb{R}$};
	\end{axis} 
	\end{tikzpicture}
			\end{center}
			\caption{Illustration for Example \ref{examplebhkasdyuiqw} with $r=1/9$, $g(x)=2x-4$ and $h(x)=x^2+2$. The vertical line illustrates the local time push $c_2=1/4$ at $x=2$.}
			\label{fig:examplebhkasdyuiqw}
\end{figure}

\begin{example}[Randomization only in terms of a Lebesgue density]\label{exampleyiweh} Let $X$ be a Wiener process, $r>0$, $f(x)= (x-1)^2 + 1$, 
\begin{align*}
g(x)=\begin{cases}
			|x|+2, &  x<0\\
			2, &  x \geq 0
		 \end{cases}
\end{align*}
and 
\begin{align*}
h(x)=\begin{cases}
			x^2+2, & x <0\\
			2, &  x \geq 0.
		 \end{cases}
\end{align*}
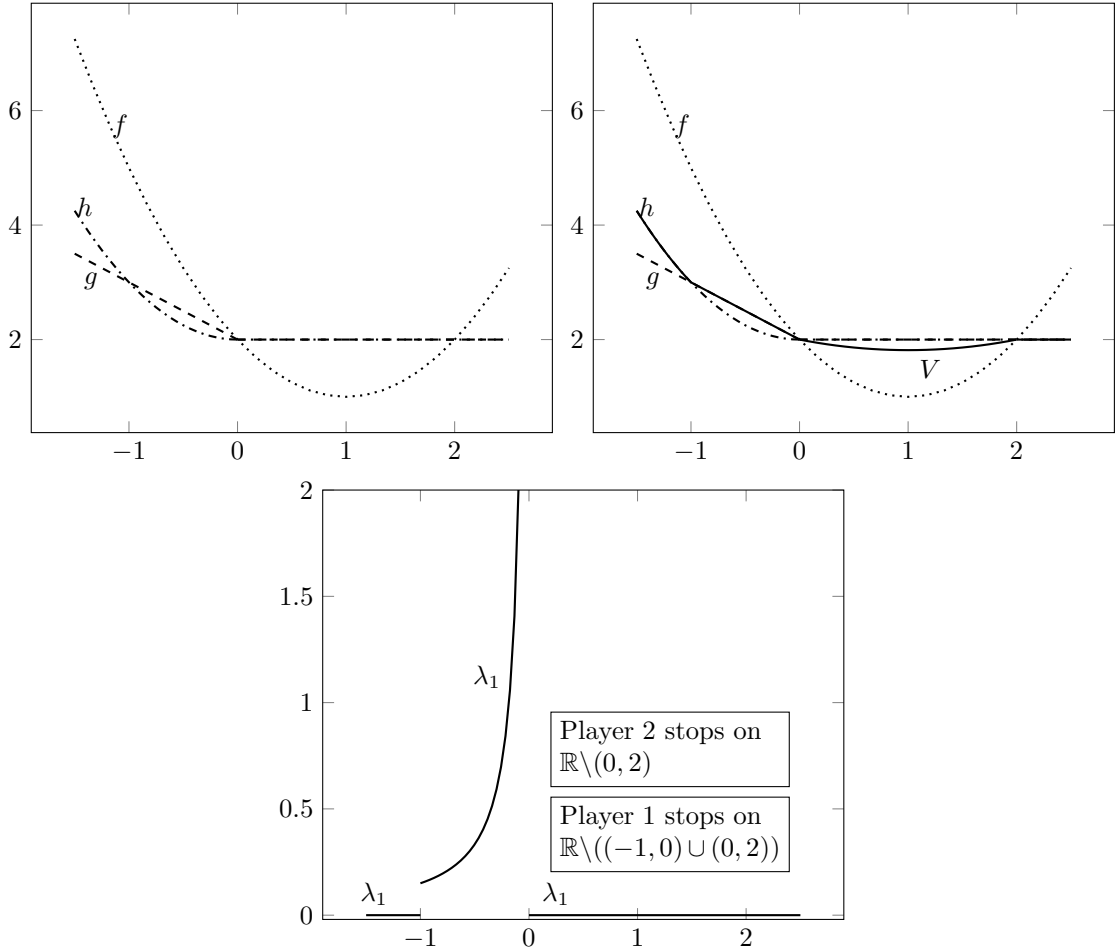
\begin{figure}[ht]
	\begin{center}
		\begin{tikzpicture}
	\begin{axis}%[ymin= -0.5, ymax= 2.5]
	\addplot[thick, dotted, domain=-1.5:2.5] {(x-1)^2+1} node [pos=0.2, above] {$f$};
	\addplot[thick, dashed, domain=-1.5:0] {abs(x)+2} node [pos=0.1, below] {$g$};
 	\addplot[thick, dashed, domain=0:2.5] {2};
        \addplot[thick, dashdotted, domain=-1.5:0] {x^2+2} node [pos=0.1, above] {$h$};
        \addplot[thick, dashdotted, domain=0:2.5] {2};
	\end{axis} 
	\end{tikzpicture}
 \begin{tikzpicture}
	\begin{axis}%[ymin= -0.5, ymax= 2.5]
	\addplot[thick, dotted, domain=-1.5:2.5] {(x-1)^2+1} node [pos=0.2, above] {$f$};
	\addplot[thick, dashed, domain=-1.5:0] {abs(x)+2} node [pos=0.1, below] {$g$};
 	\addplot[thick, dashed, domain=0:2.5] {2};
        \addplot[thick, dashdotted, domain=-1.5:0] {x^2+2} node [pos=0.1, above] {$h$};
        \addplot[thick, dashdotted, domain=0:2.5] {2};
	\addplot[thick, solid, domain=-1.5:-1] {x^2+2};
        \addplot[thick, solid, domain=-1:0] {abs(x)+2};
	\addplot[thick, solid, domain=-0:2] {
 2*(1-e^( -2*sqrt(2*0.1)))/( e^(2*sqrt(2*0.1)) -e^(-2*sqrt(2*0.1)))*e^(x*sqrt(2*0.1))
	+(2-2*(1-e^( -2*sqrt(2*0.1)))/( e^(2*sqrt(2*0.1)) -e^(-2*sqrt(2*0.1))))*e^(-x*sqrt(2*0.1))
	}  node [pos=0.6, below] {$V$};
        \addplot[thick, solid, domain=2:2.5] {2};
	\end{axis} 
	\end{tikzpicture}
			\begin{tikzpicture}
	\begin{axis}[ymin= -0.02, ymax= 2]
	\addplot[thick, solid, domain=-1.5:-1] {0} node [pos=0.15, above] {$\lambda_1$};
        \addplot[thick, solid, domain=-1:-0.01] {-0.1*(abs(x)+2)/(abs(x)+1-(x-1)^2)} node [pos=0.07, left] {$\lambda_1$};
	\addplot[thick, solid, domain=0:2.5] {0} node [pos=0.1, above] {$\lambda_1$};
		\node[draw,text width=2.9cm] at (280,8) {Player $2$ stops on $\mathbb{R}\backslash (0,2)$};
	\node[draw,text width=2.9cm] at (280,4) {Player $1$ stops on $\mathbb{R}\backslash((-1,0)\cup(0,2))$};
	\end{axis} 
	\end{tikzpicture}
			\end{center}
			\caption{
			Illustration for Example \ref{exampleyiweh} with $r=0.1$.}
			\label{fig:exampleyiweh}
\end{figure}
Note that 
${B^{g\leq h \leq f}} = (-\infty, -1] \cup \{0\} \cup [2,\infty)$,
${B^{f \lesssim h \lesssim g}}=(0,2)$ 
and ${B^{h<g<f}}=(-1,0)$. In order to find an equilibrium we start by defining
\begin{align*}
U(x):= Ce^{x\sqrt{2r}} + (2-C)e^{-x\sqrt{2r}},\enskip x \in [0,2], 
\end{align*}
where $C:=2\frac{1-e^{-2\sqrt{2r}}}{ e^{2\sqrt{2r}}-e^{-2\sqrt{2r}} }$, which implies that $U(0)=U(2)=2$. We make the additional assumption that $r>0$ is sufficiently small for it to hold that
\begin{align*}
U(x) \geq f(x),\enskip x \in [0,2]. 
\end{align*}
Using Theorem \ref{simplified-construction} and the ideas described in Remark \ref{rem:identifying} we conclude that $D_1^*= D_2^*=\mathbb{R}\backslash (0,2)$ and find a global Markovian randomized Nash equilibrium with 
\begin{align*}
V(x)=\begin{cases}
     h(x)=x^2+2, &  x \leq -1 \\
    g(x)=|x|+2, &   -1 < x \leq 0 \\
    Ce^{x\sqrt{2r}} + (2-C)e^{-x\sqrt{2r}}, &  0 < x \leq 2 \\
    h(x)=g(x)=2, &   2 < x,
\end{cases}
\end{align*}
that is attained by $\tau_{\Psi^{(2)}} = \inf\{t\geq 0: X_t \in \mathbb{R}\backslash (0,2) \}$ and 
$\tau_{\Psi^{(1)}}$ with
\begin{align*}
\Psi^{(1)}_t= \int_0^t\lambda_1(X_s)\mathbb{I}\{X_s \in (-1,0))\} ds  + 
 \infty \mathbb{I}\{\tau^{\mathbb{R}\backslash((-1,0)\cup(0,2))} \leq t \}
\end{align*}
where
\begin{align*}
\lambda_1(x)\mathbb{I}\{x \in (-1,0)\}= \left(\frac{(\mathbb{L}_X-r)g(x)}{g(x)-f(x)}\right)_+\mathbb{I}\{x \in (-1,0)\} = \frac{-r(|x|+2)}{|x|+1-(x-1)^2}\mathbb{I}\{x \in (-1,0)\}.
\end{align*}
See Figure \ref{fig:exampleyiweh} for an illustration. 
Observe that \( V \) is not differentiable at $x = -1$ and $x=2$ despite this being the case for all three payoff functions on these points. This demonstrates that a strong smooth fit principle does not necessarily apply for our games. In particular, we note that smooth fit should not generally be expected on points with $V(x)=h(x)$  where both players stop with certainty.  
\end{example}

\section{Further results on the existence of different kinds of equilibria}\label{sec:NE-condi}
Theorem \ref{thm:construction} establishes general existence of a global Markovian randomized $\epsilon$-Nash equilibrium for any $\epsilon>0$, while Theorem \ref{simplified-construction} establishes existence of a global Markovian randomized Nash equilibrium under additional assumptions.

In this section, we delve deeper into the conditions under which different types of equilibria exist. To illustrate these concepts, we begin with two examples.
The first example presents a game where only global Markovian randomized $\epsilon$-Nash equilibria with $\epsilon>0$ can be found, implying that Theorem \ref{thm:construction} cannot be made stronger without additional assumptions. 
The second example shows that the standard sufficient condition, i.e., $f(x) \leq h(x) \leq g(x), x\in {\mathcal I}$, is not a necessary condition for global Markovian pure Nash equilibrium existence.

\begin{example}[Generally only global Markovian randomized $\epsilon$-Nash equilibria exist.] \label{example:non-exist-NE}
Let $X$ be a Wiener process, $r>0$, and consider the standard optimal stopping problem 
\begin{align*}
\sup_\tau \E_x\left( e^{-r\tau}X_{\tau}^2 \right) = \E_x\left( e^{-r\tau^*}X_{\tau^*}^2 \right), x \in \mathbb{R},
\end{align*}
where we recall that the smooth fit condition allows for determining a constant $b>0$ so that $\tau^*:=\inf\{t\geq 0: X_t \notin (-b,b)\}$. Suppose $f,g$ and $h$ satisfy
\begin{align*}
g(x)> \E_x\left( e^{-r\tau^*}X_{\tau^*}^2 \right) \geq f(x)=x^2 >h(x), \enskip x \in \mathbb{R}.
\end{align*}
For the corresponding stopping game we find, relying on Proposition \ref{asso-game:characterization}, Theorem \ref{thm:construction} and Remark \ref{rem:identifying}, $D_1^*=(-\infty,b]\cup [b,\infty)$, $D_2^*= \emptyset$ and the equilibrium value 
\begin{align}\label{ex:a2esdasd}
V(x)=\E_x\left( e^{-r\tau^*}X_{\tau^*}^2 \right), \enskip x \in \mathbb{R}.
\end{align}
Indeed, using also $\mathbb{R} = {B^{h<f\leq g}}$ we find, for any $\epsilon>0$, a global Markovian randomized $\epsilon$-Nash equilibrium, with corresponding value $V$, as in \eqref{ex:a2esdasd}, that is attained by $\tau_{\Psi^{(2)}}=\infty$ and $\tau_{\Psi^{(1)}}$ with
\begin{align*}
\Psi_t^{(1)} = \int_0^t \gamma^{(1)}_\epsilon(X_s)  \mathbb{I}{\{X_s \notin (-b,b)\}}ds, 
\end{align*}
for some function $\gamma^{(1)}(\cdot)$ that is sufficiently large (cf. Corollary \ref{Equilibrium-uniqueness}). 

Furthermore, it is clear that the equilibrium $V$ cannot be attained in Markovian pure stopping times (or as the limit of Markovian pure stopping times), since if Player $1$ would stop at any given $x$ with certainty, then Player $2$ would also stop (since $h(x)<f(x)\leq V(x)$) implying that the corresponding value could not be $V(x)$. Indeed, it is easily seen that there can be no global Markovian randomized Nash equilibrium and no global Markovian pure $\epsilon$-Nash equilibrium in this example (whenever $\epsilon\geq 0$ is sufficiently small). 
\end{example}

The influential result \cite[Theorem 2.1]{ekstrom2008optimal} implies that the condition $f(x)\leq h(x) \leq g(x), x\in {\mathcal I}$ is sufficient for the existence of a global Markovian pure Nash equilibrium. Moreover, it can be be verified (cf. Theorem \ref{thm:construction}) that a weaker sufficient condition is that $f(x) \leq h(x) \leq g(x)$ or  $g(x) \leq h(x) \leq f(x)$, for each $x\in {\mathcal I}$; i.e., it is sufficient that $h(x)$ is always between $f(x)$ and $g(x)$, i.e., ${B^{h<g<f}}\cup {B^{g<f<h}} \cup {B^{f\leq g<h}}\cup {B^{h<f\leq g}} =\emptyset$. 
The following simple example shows that none of these conditions is necessary for the existence of a global Markovian pure Nash equilibrium:

\begin{example}[$f\leq h \leq g$  is not necessary for global Markovian pure Nash equilibrium existence.]\label{example-no-ordering-with-pureNE}
Let $X$ be a Wiener process and $r>0$. Let $f$ and $g$ be as in Example \ref{example:non-exist-NE}. Let $b>0$ correspond to the optimal stopping thresholds for \eqref{ex:a2esdasd}, as in Example \ref{example:non-exist-NE}. 
Suppose $h$ satisfies, for some $a \in (0,b)$,
\begin{align*}
h(x)&< g(x), \enskip  x\in \mathbb{R}\\
h(x)&< V(x), \enskip x \in(-a,a)\\
h(x)&> V(x), \enskip x \in \mathbb{R}\backslash [-a,a]\\
h(0)&< f(0).
\end{align*}
See Figure \ref{fig:example-no-ordering-with-pureNE} for an illustration. Since $h(0)< f(0)<g(0)$, we have that the mentioned sufficient conditions for global Markovian pure Nash equilibrium existence are violated. However, a global Markovian pure Nash equilibrium exists. In particular, using arguments similar to those in Example \ref{example:non-exist-NE} it can be verified that a global Markovian pure Nash equilibrium is attained 
by $\tau_{\Psi^{(1)}}=\inf\{t \geq: X_s \notin (-b,b)\}$ and $\tau_{\Psi^{(2)}}=\infty$, and that the equilibrium value $V$ is given by \eqref{ex:a2esdasd}. 
\end{example}

\begin{figure}[ht]
	\begin{center}
		\begin{tikzpicture}
	\begin{axis}[xmin= -5, xmax= 5]
	\addplot[thick, dotted, domain=-5:5] {x^2} node [pos=0.5, above] {$f$};
	\addplot[thick, dashed, domain=-5:5] {x^2+10} node [pos=0.5, above] {$g$};
	\addplot[thick, dashdotted, domain=-2:2] {8*abs(x)-5} node [pos=0.5, right] {$h$};
	\addplot[thick, dashdotted, domain=-5:-2] {x^2+7};
	\addplot[thick, dashdotted, domain=2:5] {x^2+7};
	\addplot[thick, solid, domain=-5:5] {4.618230^2/(e^(sqrt(2*0.1)*4.618230)+e^(-sqrt(2*0.1)*4.618230))*(e^(sqrt(2*0.1)*x)+e^(-sqrt(2*0.1)*x))} node [pos=0.5, above] {$V$};
	\end{axis} 
	\end{tikzpicture}
		\end{center}
			\caption{Illustration for Example \ref{example-no-ordering-with-pureNE}.}
			\label{fig:example-no-ordering-with-pureNE}
\end{figure}
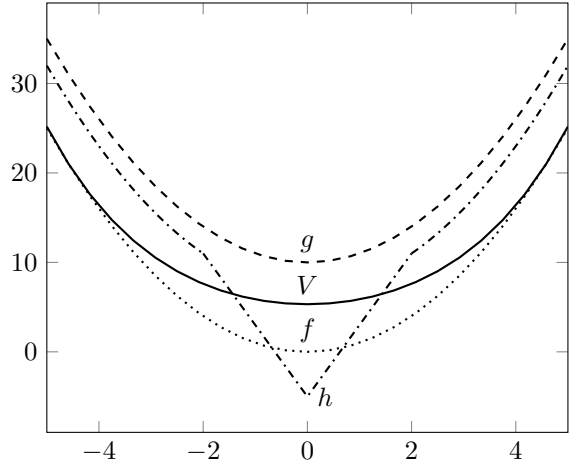

The following result allows us to determine if a global Markovian pure Nash equilibrium exists based on conditions stated in terms of the stopping sets $D^*_1$ and $D^*_2$, which for many games can be identified using Remark \ref{rem:identifying}. Note that the standard condition $f \leq h \leq g$ implies the condition of Theorem \ref{thm:NE-conditions1}, meaning that our result is more general than the currently known results (cf. Example \ref{example-no-ordering-with-pureNE}). The result follows immediately from Theorem \ref{thm:construction} (see in particular \eqref{Ne-stretegy2}) and we therefore provide no further proof. 

\begin{theorem} [Sufficient conditions for global Markovian pure Nash equilibrium existence] \label{thm:NE-conditions1}
Suppose
\begin{align*}
D_1^{*} \cap ({B^{h<g<f}} \cup {B^{h<f\leq g}})
=D_2^{*} \cap ({B^{g<f<h}} \cup {B^{f\leq g<h}}) = \emptyset,
\end{align*}
then a global Markovian pure Nash equilibrium exists. In particular, the equilibrium is attained when Player $1$ stops on  
$D_1^{*}= D_1^{*} \backslash ({B^{h<g<f}} \cup {B^{h<f\leq g}})$ and Player $2$ stops on $D_2^{*} = D_2^{*} \backslash ({B^{g<f<h}} \cup {B^{f\leq g<h}})$, without any randomization. 
\end{theorem}
An interpretation of the assumption is this: for example, if 
$D_1^{*} \cap ({B^{h<g<f}} \cup {B^{h<f\leq g}})=\emptyset$ then, for any $x$ with $h(x)<f(x) \wedge g(x)$, it holds that Player $1$ does not want to stop, so that Player $2$ has no opportunity to achieve $h(x)$, which Player $2$ would prefer over $g(x)$ in this case. We remark that Theorem \ref{thm:NE-conditions1} could have been used as an alternative starting point in the analysis of e.g., the game in Example \ref{example-no-ordering-with-pureNE}. In such an approach one could first find the value $V$ and the corresponding sets $D_i^*,i=1,2$ relying on Remark \ref{rem:identifying} and then conclude, using Theorem \ref{thm:NE-conditions1}, that the equilibrium is a global Markovian pure Nash equilibrium.  

The following result provides sufficient conditions for the absence of global Markovian pure Nash equilibria. It is based on the insight that the best-response problems of our game are standard stopping problems. 
 The proof is found in Appendix \ref{proof-app-E}.

\begin{theorem}[Sufficient conditions for the non-existence of global Markovian pure Nash equilibria]  \label{thm:NE-conditions2} 
Consider the following conditions.
\begin{enumerate}[(i)]
\item There exists a point $x_0 \in {B^{g<f<h}} \cup {B^{f\leq g<h}}=\{x \in {\cal I}: f(x) \vee g(x) < h(x)\}$ with
\begin{align}\label{asdasda}
\sup_{\tau < \tau^{({B^{g<f<h}} \cup {B^{f\leq g<h}})^c}}\mathbb{E}_{x_0}(e^{-r\tau}f(X_\tau)) > f (x_0) \vee g(x_0).
\end{align}

\item  There exists a point $x_0 \in {B^{h<g<f}} \cup {B^{h<f\leq g}}=\{x \in {\cal I}: h(x) < f(x) \wedge g(x)\}$ with
\begin{align}\label{asdasda2}
\inf_{\tau < \tau^{({B^{h<g<f}} \cup {B^{h<f\leq g}})^c}}\mathbb{E}_{x_0}(e^{-r\tau}g(X_\tau)) < f (x_0) \wedge g(x_0).
\end{align}
\end{enumerate}
If \eqref{asdasda} or \eqref{asdasda2} holds, then no global Markovian pure Nash equilibrium exists. 
\end{theorem}
We remark that Theorem \ref{thm:NE-conditions2} can be used to, e.g., see that no global Markovian pure Nash equilibrium exists in Example \ref{example1MMM}. The interpretation of, for example, (i) is as follows: On ${B^{g<f<h}} \cup {B^{f\leq g<h}}=\{x \in {\cal I}: f(x) \vee g(x) < h(x)\}$ Player 2 has no interest in simultaneous stopping, while condition \eqref{asdasda} means that Player 1 finds a stopping time on this set that dominates immediate stopping by exactly one of the players. 
We also note that results similar to Theorems \ref{thm:NE-conditions1} and \ref{thm:NE-conditions2} for a discrete time version of the present game were first reported in \cite{christensen2023markovian}.

The following example shows that the conditions in Theorem \ref{thm:NE-conditions2} are not necessary for the non-existence of global Markovian pure Nash equilibria.

\begin{example}\label{example:standardcondnotneeded} Let $X$ be a Wiener process. Let 
\begin{align*}
f(x)=4|x|, \enskip g(x)=x^2+3, \enskip &x \in (-1,1),\\
f(x)= g(x)=h(x), \enskip &x \in \mathbb{R}\backslash (-1,1). 
\end{align*}
This implies that $f(x)<g(x), x \in (-1,1)$. 
Further specify $h$ so that 
$h(x)>f(x), x \in (-1,1)$, 
$h(x)>g(x), x \in (-1,0)$, and 
$h(x)<g(x), x \in (0,1)$. 
Conclude that
${B^{g\leq h \leq f}} = \mathbb{R}\backslash (-1,1)$, 
${B^{f\leq g<h}} = (-1,0)$, and 
${B^{f \lesssim h \lesssim g}} = [0,1)$. Let $r>0$ be sufficiently small to imply that 
$(\mathbb{L}_X-r)g(x)\geq 0 $, for $x\in (-1,1)$.  See Figure \ref{figexamplestandardcondnotneeded} for an illustration.

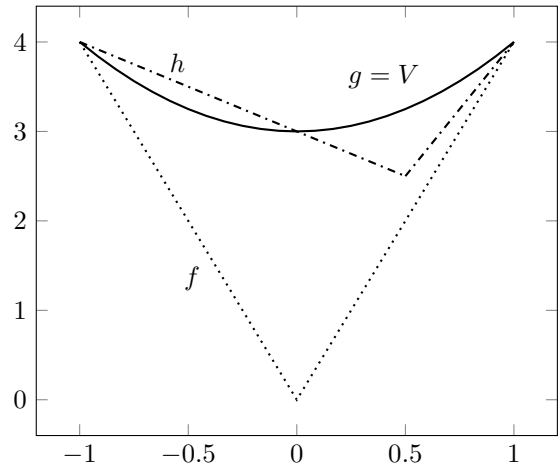
\begin{figure}[ht]
	\begin{center}
		\begin{tikzpicture}
	\begin{axis}%[ymin= -0.5, ymax= 2.5]
	\addplot[thick, dotted, domain=-1:1] {4*abs(x)} node [pos=0.3, below left] {$f$};
	\addplot[thick, solid, domain=-1:1] {x^2+3} node [pos=0.75, above left] {$g=V$};
	\addplot[thick, dashdotted, domain=-1:0.5] {3-x} node [pos=0.3, above] {$h$};
	\addplot[thick, dashdotted, domain=0.5:1] {1+3*x};
	\end{axis} 
	\end{tikzpicture}
		\end{center}
			\caption{Illustration for Example \ref{example:standardcondnotneeded}.}
			\label{figexamplestandardcondnotneeded}
\end{figure}

Using arguments similar to those in the proceeding examples we can find, for each $\epsilon>0$, a global Markovian randomized $\epsilon$-Nash equilibrium corresponding to 
$\tau_{\Psi^{(1)}}=\inf\{t \geq: X_s \in\mathbb{R}\backslash (-1,1)\}$
and $\tau_{\Psi^{(2)}}$ given by 
\begin{align*}
\Psi_t^{(2)} = \int_0^t \gamma^{(2)}_\epsilon(X_s)  \mathbb{I}{ \{X_s \in (-1,0)\}}ds + \infty \mathbb{I}\{\tau^{\mathbb{R}\backslash (-1,0)} \leq t \}
\end{align*}
for some sufficiently large function $\gamma^{(2)}_\epsilon(\cdot)$. 
Note that the corresponding equilibrium value is $V=g$, and that it cannot be attained in pure stopping times since if Player $2$ would stop in a pure way for $x\in (-1,0)$, 
then Player $1$ would deviate to obtain $h(x)>g(x)$. Moreover, it is directly verified that neither condition (i) nor (ii) of Theorem \ref{thm:NE-conditions2} holds. 
Indeed condition (ii) does not hold since ${B^{h<g<f}}\cup {B^{h<f\leq g}} = \emptyset$,  
and condition (ii) does not hold since ${B^{g<f<h}}=\emptyset$ and
\begin{align*}
\sup_{\tau < \tau^{({B^{f\leq g<h}})^c}}\mathbb{E}_{x}(e^{-r\tau}f(X_\tau)) = f (x) < g(x), \enskip x \in {B^{f\leq g<h}}=(-1,0),
\end{align*}
where the equality is verified using standard optimal stopping theory, and the inequality holds by the problem formulation. 
\end{example}
 
\appendix
 
\section{Proofs for Section \ref{sec:assoc-game}} \label{proof-app-B}

\noindent \textbf{Proof of Proposition \ref{thm:assoc-game}.} (i) By considering the time-space process $(t,X_t)$ we can apply \cite[Theorem 2.1]{ekstrom2008optimal}  to establish the result. Since our model is diffusive it also holds that $V$ is a continuous function. The result $\tilde f \leq V \leq \tilde g$ follows immediately from the equilibrium definition; compare with the arguments in (ii) below. 

(ii) ${B^{g\leq h \leq f}} \cup {B^{h<g<f}} \cup {B^{g<f<h}} \subseteq B^{g \leq f}$ is observed at the end of the Section \ref{sec:assum}. Recalling that $\tilde f= \tilde h$,  we use the notation 
\begin{align}\label{hl2q4hkl2rnv}
\tilde J (x; \tau_1, \tau_2): = 
\E_x\bigg(e^{-r(\tau_1 \wedge \tau_2)}\bigg(\tilde f(X_{\tau_1})\mathbb{I}{\{\tau_1 \leq \tau_2\}}   + \tilde g(X_{\tau_2})  \mathbb{I}{\{\tau_1 > \tau_2\}}  \bigg) \bigg),
\end{align}
for the expected reward of the associated game. It is now directly verified that the equilibrium condition (cf. \eqref{def:SE}) for the associated game is satisfied (with $\epsilon=0$) for any 
$x \in B^{g \leq f}$ when both players stop immediately on $B^{g \leq f}$; to see this use that $\tilde f(x) = \tilde g(x)$ on $B^{g \leq f}$, cf. \eqref{assoc-game-def-fgh}. 
The value of the game $V$ is moreover unique (cf. Section \ref{sec:intro}). 
This implies that $V(x) = \tilde f(x) = \tilde g(x), x \in B^{g \leq f}$ and that $B^{g \leq f} \subseteq  D_i^{*}$. 
\hfill \qedsymbol{}\\

\noindent \textbf{Proof of Proposition \ref{asso-game:characterization}.} 
(i) This proof is similar to that of \cite[Theorem 3.1]{ekstrom2017dynkin}. 
Using 
optional sampling (\cite[p. 19]{Karatzas2}), $U(X_{\tau^{D_2}})\mathbb{I}{\{ \tau^{D_2} < \infty\}}=\tilde g(X_{\tau^{D_2}})\mathbb{I}{\{ \tau^{D_2} < \infty\}}$ and $e^{-r\tau} U(X_\tau):=0$ on $\{\tau=\infty\}$ a.s., and 
$\tilde g(x)\geq \tilde f(x)$ and $U(x)\geq \tilde f(x)$, we find, for each $\tau_1 \in \mathbb{T}_1$, that
\begin{align*}
U(x) 
& \geq  \E_x\bigg(e^{-r(\tau_1 \wedge \tau^{D_2})}U (X_{\tau_1 \wedge \tau^{D_2}})  \bigg)\\
& = \E_x\bigg(e^{-r(\tau_1 \wedge \tau^{D_2})} \bigg(     U (X_{\tau_1 \wedge \tau^{D_2}} )\mathbb{I}{\{\tau^{D_2}<\infty\}}   + U(X_{\tau_1 })\mathbb{I}{\{\tau^{D_2}=\infty\}}     \bigg)  \bigg)\\
& \geq \E_x\bigg(e^{-r(\tau_1 \wedge \tau^{D_2})}\bigg(\tilde f(X_{\tau_1})\mathbb{I}{\{\tau_1 \leq \tau^{D_2}\}}   + \tilde g(X_{\tau^{D_2}})  \mathbb{I}{\{ \tau_1 > \tau^{D_2}\}}   \bigg)\mathbb{I}{\{\tau^{D_2}<\infty\}} \bigg)\\
& \enskip + \E_x\bigg(e^{-r(\tau_1 \wedge \tau^{D_2})}   \tilde f(X_{\tau_1} )\mathbb{I}{\{\tau_1<\infty\}} \mathbb{I}{\{\tau^{D_2}=\infty\}} \bigg)\\
& \geq \E_x\bigg(e^{-r(\tau_1 \wedge \tau^{D_2})}\bigg(\tilde f(X_{\tau_1})\mathbb{I}{\{\tau_1\leq \tau^{D_2}\}}   + \tilde g(X_{\tau^{D_2}})  \mathbb{I}{\{\tau_1>\tau^{D_2}\}}    \bigg) \bigg)\\
& = \tilde J (x; \tau_1, \tau^{D_2}).
\end{align*}
Hence, 
$
U(x) \geq \sup_{\tau_1 \in \mathbb{T}_1}\tilde J (x; \tau_1, \tau^{D_2}).
$
We similarly have
\begin{align*}
U(x) 
& \leq  \E_x\bigg(e^{-r(\tau_2 \wedge \tau^{D_1})}U (X_{\tau_2 \wedge \tau^{D_1}})  \bigg)\\
& = \E_x\bigg(e^{-r(\tau_2 \wedge \tau^{D_1})} \bigg(     U (X_{\tau_2 \wedge \tau^{D_1}} )\mathbb{I}{\{\tau^{D_1}<\infty\}}   + U(X_{\tau_2 })\mathbb{I}{\{\tau^{D_1}=\infty\}}     \bigg)  \bigg)\\
& \leq \E_x\bigg(e^{-r(\tau_2 \wedge \tau^{D_1})}\bigg(\tilde g(X_{\tau_2})\mathbb{I}{\{\tau_2 < \tau^{D_1}\}}   + \tilde f(X_{\tau^{D_1}})  \mathbb{I}{\{ \tau_2 \geq  \tau^{D_1}\}}   \bigg)\mathbb{I}{\{\tau^{D_1}<\infty\}} \bigg)\\
& \enskip + \E_x\bigg(e^{-r(\tau_2 \wedge \tau^{D_1})}   \tilde g(X_{\tau_2} )\mathbb{I}{\{\tau_2<\infty\}} \mathbb{I}{\{\tau^{D_1}=\infty\}} \bigg)\\
& \leq \E_x\bigg(e^{-r(\tau_2 \wedge \tau^{D_1})}\bigg(\tilde f(X_{\tau^{D_1}})\mathbb{I}{\{\tau^{D_1}\leq \tau_2\}}   + \tilde g(X_{\tau_2})  \mathbb{I}{\{\tau_2<\tau^{D_1}\}}    \bigg) \bigg)\\
& = \tilde J (x; \tau^{D_1},\tau_2).
\end{align*}
Hence,
$
U(x) \leq \inf_{\tau_2 \in \mathbb{T}_2} \tilde J (x; \tau^{D_1}, \tau_2).
$
These results imply that $U(x) \geq \tilde J (x; \tau^{D_1}, \tau^{D_2})$ and $U(x) \leq \tilde J (x; \tau^{D_1}, \tau^{D_2})$, and we may hence conclude that 
\begin{align*}
\sup_{\tau_1 \in \mathbb{T}_1}\tilde J (x; \tau_1, \tau^{D_2}) \leq U(x) = \tilde J (x; \tau^{D_1}, \tau^{D_2})\leq \inf_{\tau_2 \in \mathbb{T}_2} \tilde J (x; \tau^{D_1}, \tau_2).
\end{align*}
The result follows when comparing with Definition \ref{def:equilibrium-defs} taking uniqueness into account.

(ii) This follows from Proposition \ref{thm:assoc-game}, \eqref{assoc-game-def-fgh}, and the definitions for $B_i$ (Section \ref{sec:assum}). 
\hfill \qedsymbol{}

\section{Equilibrium construction in the general case} \label{appendix-construction}
We will here construct the remaining components of the randomized stopping times of Theorem \ref{thm:construction}, i.e., $\gamma^{(i)}_\epsilon(\cdot)$ and $\Gamma^{(i)}_\epsilon(\cdot)$, $i=1,2$. 
We remark that our construction is not unique but it rather corresponds to one way of constructing these functions.  
Note, however, that this is as expected, in view of the non-uniqueness result Corollary \ref{Equilibrium-uniqueness}(i). We already here state the following result which the subsequent analysis relies on.

\begin{lemma} \label{lemma:yiuawehjlqw}
The points in $\partial(D_1^{*}\cap ({B^{h<g<f}} \cup {B^{h<f\leq g}}))$ are separated. 
This also holds for $\partial(D_2^{*}\cap ({B^{g<f<h}} \cup {B^{f\leq g<h}}))$. 
\end{lemma}

\begin{proof}
We prove the first claim. The second claim can be proved analogously. Since ${B^{h<g<f}} \cup {B^{h<f\leq g}}$ is a union of separated intervals (cf. Assumption \ref{assum:1}), it suffices to prove that the points in 
$(\partial D_1^{*})\cap (\overline{B^{h<g<f}} \cup \overline{B^{h<f\leq g}})$  are separated. 
However, by Proposition \ref{asso-game:characterization} we have $\overline{B^{h<g<f}}\subseteq D_1^{*}$, implying that points in $(\partial D_1^{*}) \cap \overline{B^{h<g<f}}$ are in $\partial{B^{h<g<f}}$, which are separated (Assumption \ref{assum:1}). Hence, it suffices to show that the points in $(\partial D_1^{*}) \cap \overline{B^{h<f\leq g}}$ are separated. 

By Proposition \ref{asso-game:characterization} it also holds that $\overline{B^{h<f\leq g}}\cap D_1^{*}\cap D_2^{*}=B^{f=g} \subseteq D_1^*$  implying (with Assumption \ref{assum:1}) that points in $(\partial D_1^{*}) \cap \overline{B^{h<f\leq g}}\cap D_2^{*}$ are separated. Hence, it only remains to show that the points in $(\partial D_1^{*}) \cap \overline{B^{h<f\leq g}} \cap (D_2^{*})^c$ are separated.

Consider an arbitrary interval $(a,b)$ with $a,b\in \partial D_1^{*} \cap (D_2^{*})^c$, where there is a point $x\in (a,b)$ such that $x \in (D_1^{*})^c$.  Recall that $D_i^*$ is an optimal stopping set for Player $i=1,2$ (in the auxiliary game, cf. Section \ref{sec:assoc-game}). Hence, by general optimal stopping theory, we have $a,b \in (\mbox{\textnormal{supp}}(\left( (\mathbb{K}_X-r)f \right)^{+}))^c$ and that $(a,b)$ contains at least one point from $(\mbox{\textnormal{supp}}(\left( (\mathbb{K}_X-r)f \right)^{+}))$. It thus holds that each interval of the type $(a,b)$ with $a,b\in \partial D_1^{*} \cap (D_2^{*})^c$ (with a point $x\in (a,b)$ such that $x \in (D_1^{*})^c$) contains at least one point from $\partial (\mbox{\textnormal{supp}}(\left( (\mathbb{K}_X-r)f \right)^{+}))$, which are separated (Assumption \ref{assum:2}). This implies that all boundary points $a,b$ as considered above of connected components of the continuation set $(D_1^{*})^c$ are separated. 
Note that all boundary points of connected components of the continuation set, are also boundary points of stopping sets, which concludes the argument.  

\end{proof}

In the constructions we will repeatedly use the fact that there exists, for each $X_0=x \in {\cal I}$ and each fixed $d(x)>0$, a non-zero stopping time $\tau_{x,d(x)}$ such that for each stopping time 
$\tau \leq \tau_{x,d(x)}$, it holds, a.s., that 
\begin{align}\label{q24gdab}
\begin{split}
f(x)- d(x)& \leq f\left(X_\tau \right) \leq f(x)+d(x),\\
g(x)- d(x)& \leq g\left(X_\tau \right) \leq g(x)+d(x).
\end{split}
\end{align}
Indeed we attain this with $\tau_{x,d(x)}:=\inf\{t \geq 0: X_t \notin [\ubar{x}(x),\bar{x}(x)]\}$ where $\ubar{x}(x)$ and $\bar{x}(x)$ are selected so that:
\begin{itemize} 
\item $\ubar{x}(x),\bar{x}(x)\in \mathcal{I}$ and $x \in (\ubar{x}(x),\bar{x}(x))$, implying that $\tau_{x,d(x)}>0$ a.s. 
\item $\sup_{y\in[\ubar{x}(x),\bar{x}(x)]} f(y) \leq f(x)+d(x)$ and $\sup_{y\in[\ubar{x}(x),\bar{x}(x)]} g(y) \leq g(x)+d(x)$
\item $\inf_{y\in[\ubar{x}(x),\bar{x}(x)]} f(y) \geq f(x)-d(x)$ and $\inf_{y\in[\ubar{x}(x),\bar{x}(x)]} g(y) \geq g(x)-d(x)$. 
\end{itemize}
One can directly see that selecting $\ubar{x}(x)$ and $\bar{x}(x)$ so that these conditions are satisfied is possible, using only the continuity of $f$, $g$ and the paths of $(X_t)$.

\begin{construction}[Equilibrium construction for Player $2$] \label{construction1}
In order for the equilibrium stopping time of Player $2$ to be fully specified we have to define $\gamma_\epsilon^{(2)} $ and $\Gamma^{(2)}_\epsilon$ in \eqref{Ne-stretegy2}. 

We begin by defining
(i) $\gamma_\epsilon^{(2)}:{\cal I}\rightarrow [0,\infty)$ as a measurable locally bounded function that takes the value zero outside of $D_2^{*}\cap({B^{g<f<h}} \cup {B^{f\leq g<h}})  \backslash \mbox{iso}(D_2^{*}\cap({B^{g<f<h}} \cup {B^{f\leq g<h}}))$, and 
(ii) $\Gamma^{(2)}_\epsilon:\mbox{iso}(D_2^{*}\cap({B^{g<f<h}} \cup {B^{f\leq g<h}}))\rightarrow [0,\infty)$.

$\bullet$ Definition of $\gamma_\epsilon^{(2)}$ and $\Gamma^{(2)}_\epsilon$ on 
$D_1^{*} \cap D_2^{*} \cap ({B^{g<f<h}} \cup {B^{f\leq g<h}})$. Here define 
\begin{align*}
d(x)= 
\begin{cases}
			\epsilon/4, &  \mbox{ if } f(x)+\epsilon/4 \geq 0\\
			\epsilon/8, &  \mbox{ if } f(x)+\epsilon/4 <0,
\end{cases}
\end{align*}
and
\begin{align*}
k(x)= 
\begin{cases}
			\infty, &  \mbox{ if } f(x)+\epsilon/4 \geq 0\\
			\frac{1}{r}\log\left(\frac{f(x)+\epsilon/8}{f(x)+\epsilon/4}\right), &  \mbox{ if } f(x)+\epsilon/4 <0,
\end{cases}
\end{align*}
so that $k(x)>0$, and $\left(f(x)+\epsilon/8\right)e^{-rk(x)}= f(x)+\epsilon/4$ whenever $f(x)+\epsilon/4<0$. This implies that 
\begin{align}\label{q24gdab2}
\max\left(f(x) + d(x); e^{-rk(x)}(f(x) + d(x))\right)=f(x)+\epsilon/4.
\end{align}

\begin{enumerate}[(i)]
\item Consider any specific 
$x \in \mbox{iso}(D_1^{*} \cap D_2^{*}\cap ({B^{g<f<h}} \cup {B^{f\leq g<h}}))$. Let $\tau:=\inf\{t \geq 0: \Gamma_\epsilon^{(2)}(x)l_t^x\geq E_2\}$ (which we note is a stopping time of the kind \eqref{def:admissST-eq}--\eqref{eq:psi-strat}, that depends directly on $X_0=x$). Relying on this notation for $\tau$ we set  $\Gamma_\epsilon^{(2)}(x)<\infty$ to be a value such that the condition
\begin{align}\label{asfqer32r4}
\mathbb{P}_x\left(\tau > \tau_{x,d(x)}\wedge k(x)\right) \left(\max_{y\in [\ubar{x}(x),\bar{x}(x)]}M(y) - f(x) - \epsilon/4\right) \leq \epsilon/4,
\end{align}
is satisfied (here $M(x)$ is the continuous and therefore in particular locally bounded function defined in \eqref{assum:integrability} and $\tau_{x,d(x)}$ is defined in the introduction of this appendix); to see that this is possible observe that $\mathbb{P}_x\left(\tau > \tau_{x,d(x)}\wedge k(x)\right)\searrow 0$ as $\Gamma_\epsilon^{(2)}(x)\rightarrow \infty$.

\item Consider the set $D_1^{*} \cap D_2^{*} \cap ({B^{g<f<h}} \cup {B^{f\leq g<h}}) \backslash \mbox{iso}(D_1^{*} \cap D_2^{*}\cap ({B^{g<f<h}} \cup {B^{f\leq g<h}}) )$. Here we re-define the stopping time $\tau$ according to $\tau:=\inf\{t \geq 0: \int_0^t \gamma_\epsilon^{(2)}(X_s)ds \geq E_2\}$ and relying on this definition we select the function $\gamma_\epsilon^{(2)}$ so that \eqref{asfqer32r4} holds for all $x$ in the considered set. A construction of such a function $\gamma_\epsilon^{(2)}$  is as follows: 
\begin{itemize}
\item[-] Represent the considered set as a countable union of bounded intervals $I_n$, where the measure of each interval is dominated by a strictly positive constant (cf. Lemma \ref{lemma:yiuawehjlqw}). 
\item[-] Consider the stopping times 
$\tau(n):=\inf\{t \geq 0: \int_0^t C_n\mathbb{I}\{X_s \in I_n\} ds \geq E_2\}$, where $C_n>0$. 
\item[-] Then the function $(x,C_n)\in \bar{I_n} \times(0,\infty)\mapsto p(x,C_n;\epsilon):=\mathbb{P}_x\left(\tau(n) > \tau_{x,d(x)}\wedge k(x)\right)$ is continuous and bounded, with $p(x,C_n;\epsilon) \searrow0$ as $C_n\rightarrow \infty$.  Hence, for any fixed $\epsilon>0$ there exists a constant $C_n>0$ such that \eqref{asfqer32r4} holds if we replace $\tau$ with $\tau(n)$, for each $x\in I_n$. 
\item[-] Now specify the function $\gamma_\epsilon^{(2)}$ so that it is continuous on each (bounded) $I_n$ with $\gamma_\epsilon^{(2)}(x) \geq C_n, x \in I_n$, for each $n$. 
\end{itemize}

\end{enumerate}
$\bullet$ Definition of $\gamma_\epsilon^{(2)}$ and $\Gamma^{(2)}_\epsilon$ on  $ (D_1^{*})^c \cap D_2^{*}\cap {B^{f\leq g<h}}$ (which we note equals 
$(D_1^{*})^c\cap D_2^{*} \cap ({B^{g<f<h}} \cup {B^{f\leq g<h}})$).  
Define 
\begin{align*}
d_1(x)= 
\begin{cases}
			\epsilon/4, &  \mbox{ if } g(x)+\epsilon/4 \geq 0\\
			\epsilon/8, &  \mbox{ if } g(x)+\epsilon/4 <0,
\end{cases}
\end{align*}
and 
\begin{align*}
k_1(x)= 
\begin{cases}
			\infty, &  \mbox{ if } g(x)+\epsilon/4 \geq 0\\
			\frac{1}{r}\log\left(\frac{g(x)+\epsilon/8}{g(x)+\epsilon/4}\right), &  \mbox{ if } g(x)+\epsilon/4 <0,
\end{cases}
\end{align*}
so that
\begin{align*}
\max\left(g(x) + d_1(x); e^{-rk_1(x)}(g(x) + d_1(x))\right)= g(x)+\epsilon/4.
\end{align*}
Define 
\begin{align*}
d_2(x)= 
\begin{cases}
			\epsilon/4, &  \mbox{ if } g(x)-\epsilon/4 \leq 0\\
			\epsilon/8, &  \mbox{ if } g(x)-\epsilon/4 >0,
\end{cases}
\end{align*}
and
\begin{align*}
k_2(x)= 
\begin{cases}
			\infty, &  \mbox{ if } g(x)-\epsilon/4 \leq 0\\
			\frac{1}{r}\log\left(\frac{g(x)-\epsilon/8}{g(x)-\epsilon/4}\right), &  \mbox{ if } g(x)-\epsilon/4 >0,
\end{cases}
\end{align*}
so that
\begin{align*}
\min\left(e^{-rk_2(x)}(g(x) - d_2(x));g(x) - d_2(x)\right)= g(x)-\epsilon/4.
\end{align*}

\begin{enumerate}[(i)]
\item Consider any specific $x \in \mbox{iso}(  (D_1^{*})^c \cap D_2^{*} \cap  {B^{f\leq g<h}} )$. Let 
$\tau:=\inf\{t\geq 0: \Gamma_\epsilon^{(2)}(x)l_t^x\geq E_2\}$ 
and set $\Gamma_\epsilon^{(2)}(x)<\infty$ to a value so that 
\begin{align}\label{qe21q3edaDD}
\begin{split}
\mathbb{P}_x\left(\tau > \tau_{x,d_1(x)}\wedge k_1(x) \right)  \left(\max_{y\in [\ubar{x}(x),\bar{x}(x)]}M(y) - g(x) - \epsilon/4\right)&\leq \epsilon/4\\
\mathbb{P}_x\left(\tau > \tau_{x,d_2(x)} \wedge k_2(x) \wedge \frac{1}{2}\inf\{t: X_t \in D_1^{*}\}\right) \left(\min_{y\in [\ubar{x}(x),\bar{x}(x)]}(-M(y))- (g(x) - \epsilon/4)\right) &\geq -\epsilon/4.
\end{split}
\end{align}

\item Consider the set $(D_1^{*})^c \cap D_2^{*}\cap {B^{f\leq g<h}}  \backslash \mbox{iso}( (D_1^{*})^c \cap D_2^{*}\cap {B^{f\leq g<h}})$. Here let 
$\tau:=\inf\{t\geq 0: \int_0^t \gamma_\epsilon^{(2)}(X_s)ds\geq E_2\}$ and define the function $\gamma_\epsilon^{(2)}$ so that \eqref{qe21q3edaDD} holds for all $x$ in the considered set. This can be done analogously to the construction of $\gamma_\epsilon^{(2)}$ on  $D_1^{*} \cap D_2^{*} \cap ({B^{g<f<h}} \cup {B^{f\leq g<h}}) \backslash \mbox{iso}(D_1^{*} \cap D_2^{*}\cap ({B^{g<f<h}} \cup {B^{f\leq g<h}}))$ (above). 
\end{enumerate}
 
\end{construction}

Note that our game is anti-symmetric in the sense that Player $2$ wants to minimize a reward that Player $1$ wants to maximize. For a suitable re-specification of the payoff functions we can thus essentially have Player $1$ being the minimizer and Player $2$ being the maximizer without changing anything else in the framework. This allows us to describe the equilibrium construction for Player $1$ by a reference to the equilibrium construction for Player $2$. The detailed argument is provided here: 

\begin{construction}[Equilibrium construction for Player $1$] \label{construction2} 
It is clear that our game---where Player $1$ is a maximizer and Player $2$ is a minimizer for \eqref{cost-function}---is equivalent to a game of the same type where Player $2$ is the maximizer and Player $1$ is the minimizer for the reward
\begin{align*}
-\E_x\left(e^{-r(\tau_1 \wedge \tau_2)}\left(f(X_{\tau_1})\mathbb{I}{\{\tau_1<\tau_2\}} + g(X_{\tau_2})\mathbb{I}{\{\tau_1>\tau_2\}} + h(X_{\tau_1})\mathbb{I}{\{\tau_1=\tau_2\}}\right)\right).
\end{align*}
Note that this corresponds to Player $2$ (who is the maximizer in this alternative formulation) 
receiving $-g$ if stopping first,
receiving $-f$ if not stopping first, and
receiving $-h$ if stopping at the same time as Player $1$. 
 
Relying on these observations it is clear that in order to construct $\gamma_\epsilon^{(1)}$ and $\Gamma^{(1)}_\epsilon$ in \eqref{Ne-stretegy2} for the equilibrium stopping time for Player $1$, one can simply follow the steps of Construction \ref{construction1} and replacing
$g$ with $-f$,
$f$ with $-g$,
$D_2^*$ with $D_1^*$,
$D_1^*$ with $D_2^*$, 
${B^{f\leq g<h}}$ with ${B^{h<f\leq g}}$, and
${B^{g<f<h}}$ with ${B^{h<g<f}}$. 
\end{construction}

\section{Proofs for Section \ref{sec:NE-exist}} \label{proof-app-C}

The proof of Theorem \ref{thm:construction} relies crucially on the following result which says that the continuous additive functionals of the equilibrium stopping times do not blow up and that simultaneous stopping is possible only when both players stop with certainty.

\begin{proposition}\label{prop:simul-stopping} Let $\Psi_t^{(i)},A_t^{(i)},i=1,2$ be as in Theorem \ref{thm:construction}. 
\begin{enumerate}[(i)]
\item It holds that $A_t^{(1)},A_t^{(2)}<\infty$ for  $t<\infty$, $\mathbb{P}_x$-a.s. 
\item Consider an arbitrary fixed $\tau_1 \in \mathbb{T}_1$. If $\tau_1(\omega)=\tau_{\Psi^{(2)}}(\omega)<\infty$ then $X_{\tau_1}(\omega)\in \overline{D_2^* \backslash ({B^{g<f<h}}\cup {B^{f\leq g<h}})}$, $\mathbb{P}_x$-a.s. Similarly, consider an arbitrary fixed $\tau_2 \in \mathbb{T}_2$. If $\tau_2(\omega)=\tau_{\Psi^{(1)}}(\omega)<\infty$ then $X_{\tau_2}(\omega)\in \overline{D_1^* \backslash ({B^{h<g<f}}\cup {B^{h<f\leq g}})}$, $\mathbb{P}_x$-a.s. 
\end{enumerate}
\end{proposition}

\begin{proof}
This follows from \cite[Lemma 4.4]{christensen2024existence}. 
\end{proof}

\noindent\textbf{Proof of Theorem \ref{thm:construction}.} For ease of exposition we write in this proof the proposed equilibrium stopping time pair in the theorem statement as $(\tau^*_1,\tau^*_2):=(\tau_{\Psi^{(1)}},\tau_{\Psi^{(2)}})$. 

Fix an arbitrary $\epsilon>0$. In the items below we verify that the equilibrium condition for Player $1$ holds for the proposed equilibrium stopping time pair; i.e., we prove that the first inequality in the equilibrium condition \eqref{def:SE} holds. In view of Construction \ref{construction2} it is clear that the corresponding result for Player $2$ also holds. This corresponds to the second part of Theorem \ref{thm:construction}(ii). 

From the items below it also follows that the equilibrium value $J(\cdot;\tau^*_1,\tau^*_2)$ coincides with the value of the associated game $V$ in Proposition \ref{asso-game:characterization} (in applicable cases $V(x)$ is obtained by sending $\epsilon \searrow 0$). This corresponds to the remaining statements of Theorem \ref{thm:construction}.

Equation \eqref{asdasdDd} and Proposition \ref{prop:simul-stopping} are used repeatedly.

\begin{enumerate}

\item Consider an arbitrary point $x \in {B^{g\leq h \leq f}}$. The proposed equilibrium, i.e., $(\tau_1^*,\tau^*_2)$, corresponds to both players stopping immediately. The value corresponding to the proposed equilibrium is therefore $J(x;\tau^*_1,\tau^*_2)=h(x)$. 
If Player $1$ deviates from $(\tau_1^*,\tau^*_2)$ by using a different stopping time $\tau_1 \in \mathbb{T}_1$, 
then
either 
$J(x;\tau_1,\tau^*_2)=h(x)$ 
or $J(x;\tau_1,\tau^*_2)=g(x)\leq h(x)$; use the definition of ${B^{g\leq h \leq f}}$ to see this. 
Hence, the first part of \eqref{def:SE} holds, for each $\epsilon  \geq 0$.

\item Consider $x \in {B^{h<g<f}}$. The proposed equilibrium corresponds to Player $2$ stopping immediately, while Player $1$ randomizes. 
The corresponding value is $J(x;\tau^*_1,\tau^*_2)=g(x)$. 
If Player $1$ deviates by using a different stopping time $\tau_1 \in \mathbb{T}_1$, 
then 
either $J(x;\tau_1,\tau^*_2)=g(x)$ 
or $J(x;\tau_1,\tau^*_2)=h(x)<g(x)$; use the definition of ${B^{h<g<f}}$ to see this. 
Hence the first part of \eqref{def:SE} holds, for each $\epsilon  \geq 0$.

\item In the items below we consider the points $x \in {B^{f \lesssim h \lesssim g}} \cup {B^{f\leq g<h}} \cup {B^{h<f\leq g}} \subseteq B^{f\leq g}$.  We rely on the observation that
$$x \in ({B^{f \lesssim h \lesssim g}} \cup {B^{f\leq g<h}} \cup {B^{h<f\leq g}})  \cap D_1^{*} \cap D_2^{*} \Rightarrow f(x)=g(x).$$
To see that this holds use that $D_i^{*}$ are the stopping sets of the associated game, that 
$\tilde f = f = \tilde h \leq \tilde g = g$ on ${B^{f \lesssim h \lesssim g}} \cup {B^{f\leq g<h}} \cup {B^{h<f\leq g}}$ and basic observations inline with the items above; essentially, the observation is that simultaneous stopping, yielding $\tilde h(x)$ for the associated game, can, in case $f(x) \leq g(x)$, only occur if $f(x) = g(x)$. This implies also that 
\begin{align}\label{hwqejlkaw}
 {B^{f \lesssim h \lesssim g}}    \cap D_1^{*} \cap D_2^{*} = \emptyset.
\end{align}

$\bullet$ For $x \in   {B^{f \lesssim h \lesssim g}} \cap  D_2^{*} $. The proposed equilibrium says that Player $2$ should stop (without randomization) and that Player $1$ should not do so (cf. \eqref{Ne-stretegy}--\eqref{Ne-stretegy2} and \eqref{hwqejlkaw}). 
Hence, $J(x;\tau_1^*,\tau_2^*)=g(x)$.
If Player $1$ deviates by using another stopping time $\tau_1$ then either $J(x;\tau_1,\tau_2^*)=g(x)$ or $J(x;\tau_1,\tau_2^*)=h(x)\leq g(x)$, where the inequality holds by definition of ${B^{f \lesssim h \lesssim g}}$.
Hence, the first part of \eqref{def:SE} holds, for each $\epsilon  \geq 0$.

$\bullet$ For $x \in  {B^{h<f\leq g}} \cap D_2^{*}  $. The proposed equilibrium says that Player $2$ should stop (without randomization) and that Player $1$ should not do so (cf. \eqref{Ne-stretegy}--\eqref{Ne-stretegy2}). Hence, $J(x;\tau_1^*,\tau_2^*) = g(x)$. If Player $1$ deviates by using another stopping time $\tau_1$ then either $J(x;\tau_1,\tau_2^*)=g(x)$ or $J(x;\tau_1,\tau_2^*)=h(x)< g(x)$ (cf. definition of ${B^{h<f\leq g}}$). 
Hence, the first part of \eqref{def:SE} holds, for each $\epsilon  \geq 0$.

$\bullet$  For $x \in {B^{f\leq g<h}} \cap D_1^{*} \cap D_2^{*}  $. The proposed equilibrium says that Player $2$ should approximate stopping by using randomization and that Player $1$ should stop without randomization. Using this and that $x \in {B^{f\leq g<h}} \cap D_1^{*} \cap D_2^{*} \Rightarrow f(x)=g(x)$ we obtain $J(x;\tau_1^*,\tau^*_2) + \epsilon = g(x) + \epsilon = f(x) + \epsilon$. In order to show that the first part of \eqref{def:SE} holds, for a fixed $\epsilon>0$, it is enough to show that
\begin{align}\label{adqwefas}
J(x;\tau_1,\tau^*_2) \leq  J(x;\tau_1^*,\tau^*_2) + \epsilon/2 = g(x) + \epsilon/2 =  f(x) + \epsilon/2, \enskip \mbox{for all $\tau_1$}.
\end{align}
(Note that it would here suffice to show that \eqref{adqwefas} holds with $\epsilon/2$ replaced with $\epsilon$, and without mentioning the last equality; but we rely on the stronger result \eqref{adqwefas} later in the proof.)

Let us first establish that 
\begin{align}\label{234rasf}
\begin{split}
& J(x;\tau_1,\tau^*_2)\\
& := \E_x\left(e^{-r(\tau_1 \wedge \tau_2^*)}\left(f(X_{\tau_1})\mathbb{I}{\{\tau_1<\tau_2^*\}} + g(X_{\tau_2^*})\mathbb{I}{\{\tau_1> \tau_2^*\}} + h(X_{\tau_1})\mathbb{I}{\{\tau_1=\tau_2^*\}} \right)\right)\\
& = \E_x\left(e^{-r(\tau_1 \wedge \tau_2^*)}\left(f(X_{\tau_1})\mathbb{I}{\{\tau_1<\tau_2^*\}} + g(X_{\tau_2^*})\mathbb{I}{\{\tau_1 > \tau_2^*\}}+ h(X_{\tau_1})\mathbb{I}{\{\tau_1=\tau_2^*\}}\right) \mathbb{I}{\{\tau_2^* \leq \tau_{x,d(x)}\wedge k(x)\}} \right)\\
&\enskip + \E_x\left(e^{-r(\tau_1 \wedge \tau_2^*)}\left(f(X_{\tau_1})\mathbb{I}{\{\tau_1<\tau_2^*\}} + g(X_{\tau_2^*})\mathbb{I}{\{\tau_1 > \tau_2^*\}}+ h(X_{\tau_1})\mathbb{I}{\{\tau_1=\tau_2^*\}}\right) \mathbb{I}{\{\tau_2^* > \tau_{x,d(x)}} \wedge k(x)\} \right)\\
&\leq \mathbb{P}_x\left(\tau_2^* \leq \tau_{x,d(x)} \wedge k(x)\right) \max\left(f(x) + d(x); e^{-rk(x)}(f(x) + d(x));g(x) + d(x); e^{-rk(x)}(g(x) + d(x))\right) \\ 
&\enskip + \mathbb{P}_x\left(\tau_2^*> \tau_{x,d(x)} \wedge k(x) \right)\max_{y\in [\ubar{x}(x),\bar{x}(x)]}M(y)\\
& = \mathbb{P}_x\left(\tau_2^* \leq \tau_{x,d(x)} \wedge k(x)\right)\left(f(x) + \epsilon/4 \right)  + \mathbb{P}_x\left(\tau_2^*> \tau_{x,d(x)}\wedge k(x)\right)\max_{y\in [\ubar{x}(x),\bar{x}(x)]}M(y)\\
&=f (x) + \epsilon/4   +  \mathbb{P}_x\left(\tau_2^*> \tau_{x,d(x)}\wedge k(x)\right)\left(\max_{y\in [\ubar{x}(x),\bar{x}(x)]}M(y) - f(x) - \epsilon/4 \right).
\end{split}
\end{align}
The inequality above follows from the strong Markov property and the definition of $\tau_{x,d(x)}$ in Construction \ref{construction1} (see e.g., \eqref{q24gdab}) together with the fact that simultaneous stopping cannot not occur unless $h\leq f$ or $h\leq g$ (since Player $2$ randomizes on ${B^{g<f<h}} \cup {B^{f\leq g<h}}$). The equality after that follows from \eqref{q24gdab2} and $f(x)=g(x)$ (for the present $x$). The remaining steps are easily verified.

Relying on Construction \ref{construction1} and the observation that each stopping time $\tau$ there satisfies $\tau \geq  \tau_2^*$ a.s. (by construction of $\tau_2^*$) we obtain
\begin{align*}
\mathbb{P}_x\left(\tau_2^*> \tau_{x,d(x)}\wedge k(x)\right)\left( \max_{y\in [x[\ubar{x}(x),\bar{x}(x)]}M(y)- f(x) - \epsilon/4 \right) \leq \epsilon/4.
\end{align*}
This inequality and \eqref{234rasf} implies that \eqref{adqwefas} holds and we are done.

$\bullet$  For $x \in  {B^{f\leq g<h}} \cap (D_1^{*})^c  \cap D_2^{*}  $. The proposed equilibrium says that Player $2$ should approximate stopping using randomization and that Player $1$ should not stop and not randomize. In order to show that the first part of \eqref{def:SE} holds, for the fixed $\epsilon>0$, we should show that
\begin{align} \label{adqwefas2}
J(x;\tau_1,\tau^*_2) \leq  J(x;\tau_1^*,\tau^*_2) + \epsilon, \enskip \mbox{for all $\tau_1$}.
\end{align}
Using arguments similar to those that lead to \eqref{234rasf} (with $d_1(x)$ and $k_1(x)$ instead of $d(x)$ and $k(x)$), and that $f\leq g$ on ${B^{f\leq g<h}}$, it can be shown that 
\begin{align*}
J(x;\tau_1,\tau^*_2) \leq g(x) + \epsilon/2, \enskip \mbox{for all $\tau_1$}.
\end{align*}
Hence, if we can show that
\begin{align}\label{adqwefas3}
J(x;\tau^*_1,\tau^*_2) \geq g(x)- \epsilon/2,
\end{align}
then \eqref{adqwefas2} holds and we are done. It holds that
\begin{align*}
J(x;\tau_1^*,\tau^*_2)
&=\E_x\left(e^{-r(\tau_1^* \wedge \tau_2^*)}\left(f(X_{\tau_1^*})\mathbb{I}_{\{\tau_1^*<\tau_2^*\}} + g(X_{\tau_2^*})\mathbb{I}_{\{\tau_1^*> \tau_2^*\}} + h(X_{\tau_1^*})\mathbb{I}_{\{\tau_1^*=\tau_2^*\}} \right)\right)\\
&\geq \mathbb{P}_x\left(\tau_2^* \leq \tau_{x,d_2(x)} \wedge k_2(x) \wedge \frac{1}{2}\inf\{t: X_t \in D_1^{*}\}\right)\min(e^{-rk_2(x)}(g(x) - d_2(x));g(x) - d_2(x))\\ 
&\enskip + \mathbb{P}_x\left(\tau_2^* > \tau_{x,d_2(x)} \wedge k_2(x) \wedge \frac{1}{2}\inf\{t: X_t \in D_1^{*}\}\right) \min_{y\in [\ubar{x}(x),\bar{x}(x)]}(-M(y))\\
& = \mathbb{P}_x\left(\tau_2^* \leq \tau_{x,d_2(x)} \wedge k_2(x) \wedge \frac{1}{2}\inf\{t: X_t \in D_1^{*}\}\right)(g(x) - \epsilon/4)\\ 
&\enskip + \mathbb{P}_x\left(\tau_2^* > \tau_{x,d_2(x)} \wedge k_2(x) \wedge \frac{1}{2}\inf\{t: X_t \in D_1^{*}\}\right) \min_{y\in [\ubar{x}(x),\bar{x}(x)]}(-M(y))\\
& = g(x) - \epsilon/4 \\
&\enskip + \mathbb{P}_x\left(\tau_2^* > \tau_{x,d_2(x)} \wedge k_2(x) \wedge \frac{1}{2}\inf\{t: X_t \in D_1^{*}\}\right) \left(\min_{y\in [\ubar{x}(x),\bar{x}(x)]}(-M(y))- (g(x) - \epsilon/4)\right),
\end{align*}
which can be seen using arguments similar to those in the item above and the observations that $(D_1^{*})^c$ is an open set and that if stopping occurs before $\inf\{t: X_t \in D_1^{*}\}$ then the payoff must be given by $g$ (since Player $1$ does not stop and does not randomize on $(D_1^{*})^c$). 

Relying on Construction \ref{construction1}  and the observation that each stopping time $\tau$ there satisfies $\tau \geq  \tau_2^*$ a.s. we obtain
\begin{align*}
\mathbb{P}_x\left(\tau_2^* > \tau_{x,d_2(x)} \wedge k_2(x) \wedge \frac{1}{2}\inf\{t: X_t \in D_1^{*}\}\right) \left(\min_{y\in [\ubar{x}(x),\bar{x}(x)]}(-M(y))- (g(x) - \epsilon/4)\right) \geq -\epsilon/4.
\end{align*}
It follows that \eqref{adqwefas3} holds and we are done.

$\bullet$ For $x\in ({B^{f \lesssim h \lesssim g}} \cup {B^{f\leq g<h}} \cup {B^{h<f\leq g}}) \cap (D_2^{*})^c$. If we can show that the following inequality holds then we are done 
\begin{align} \label{afasfqwreqrqwrf}
\sup_{\tau_1 \in \mathbb{T}_1} J(x;\tau_1,\tau^*_2) - \epsilon \leq J(x;\tau^*_1,\tau^*_2).
\end{align}
Let $(a,b)$ be the largest open interval so that $x \in (a,b)$ and $(a,b) \subseteq (D_2^{*})^c$. 
Note that $(D_2^{*})^c \subseteq  {B^{f \lesssim h \lesssim g}} \cup {B^{f\leq g<h}} \cup {B^{h<f\leq g}}$. 
It follows that $a,b \in \partial D_2^{*} \cap \overline{{B^{f \lesssim h \lesssim g}} \cup {B^{f\leq g<h}} \cup {B^{h<f\leq g}}}$ in case $a$ and $b$, respectively are finite; 
whereas 
if $a=-\infty$ then $y \in \overline{{B^{f \lesssim h \lesssim g}} \cup {B^{f\leq g<h}} \cup {B^{h<f\leq g}}}$ for all $y\leq x$, and 
if $b=\infty$ then $y \in \overline{{B^{f \lesssim h \lesssim g}} \cup {B^{f\leq g<h}} \cup {B^{h<f\leq g}}}$ for all $y\geq x$. 

Recall that $\tau^{(a,b)^c}:= \inf\{t \geq 0: X_t \notin (a,b)\}$. Consider the associated game (Section \ref{sec:assoc-game}) and observe:

\begin{enumerate}[(i)]

\item For the present starting value $X_0=x$ it holds, keeping the definition of $\tilde \tau_2^*$ in mind, that $\tau^{(a,b)^c}= \tilde \tau_2^*$ a.s. and that $\tilde \tau_2^*$ will occur at the latest when the state process leaves $B^{f \leq g}\supseteq {B^{f \lesssim h \lesssim g}} \cup {B^{f\leq g<h}} \cup {B^{h<f\leq g}}$. Indeed it holds a.s. that  
\begin{align*}
e^{-r(\tau_1 \wedge \tau^{(a,b)^c})} \tilde f(X_{\tau_1})\mathbb{I}{\{\tau_1\leq\tau^{(a,b)^c}\}}&=e^{-r(\tau_1 \wedge \tau^{(a,b)^c})} f(X_{\tau_1})\mathbb{I}{\{\tau_1 \leq \tau^{(a,b)^c}\}},\\
e^{-r(\tau_1 \wedge \tau^{(a,b)^c})} \tilde  g(X_{\tau^{(a,b)^c}}) &= e^{-r(\tau_1 \wedge \tau^{(a,b)^c})}g(X_{\tau^{(a,b)^c}}).
\end{align*}

\item Simultaneous stopping occurs at any specific $x$ for the equilibrium of the associated game $(\tilde \tau_1^*,\tilde \tau_2^*)$ (cf. Proposition \ref{thm:assoc-game}) if and only if $\tilde f(x)=\tilde g(x)$.

\end{enumerate}

Using \eqref{hl2q4hkl2rnv} and the observations above we find that the equilibrium value of the associated game can, for the present $x$, be written as
\begin{align}\label{2311212}
\begin{split}
& \tilde J(x; \tilde \tau_1^*,\tilde \tau_2^*)\\
&= \sup_{\tau_1 \in \mathbb{T}_1} \tilde J(x; \tau_1,\tilde \tau_2^*) \\
&= \sup_{\tau_1 \in \mathbb{T}_1} \E_x\bigg(e^{-r(\tau_1 \wedge \tau^{(a,b)^c})}\bigg(\tilde f(X_{\tau_1})\mathbb{I}{\{\tau_1 \leq \tau^{(a,b)^c}\}}   + \tilde g(X_{\tau^{(a,b)^c}})  \mathbb{I}{\{\tau_1 > \tau^{(a,b)^c}\}}\bigg) \bigg)\\
&= \sup_{\tau_1 \in \mathbb{T}_1} \E_x\bigg(e^{-r(\tau_1 \wedge \tau^{(a,b)^c})}\bigg(f(X_{\tau_1})\mathbb{I}{\{\tau_1\leq\tau^{(a,b)^c}\}}   + g(X_{\tau^{(a,b)^c}})  \mathbb{I}{\{\tau_1 > \tau^{(a,b)^c}\}} \bigg) \bigg)\\
&= \sup_{\tau_1 \in \mathbb{T}_1} \E_x\bigg(e^{-r(\tau_1 \wedge \tau^{(a,b)^c})}\bigg(f(X_{\tau_1})\mathbb{I}{\{\tau_1<\tau^{(a,b)^c}\}}   + g(X_{\tau^{(a,b)^c}})  \mathbb{I}{\{\tau_1 \geq \tau^{(a,b)^c}\}} \bigg) \bigg).
\end{split}
\end{align}
Note that \eqref{2311212} is based on $\tilde \tau_2^*$ (equilibrium stopping time for Player $2$ for the associated game), whereas 
\eqref{afasfqwreqrqwrf} is based on  $\tau_2^*$  (the proposed equilibrium stopping time for Player $2$ for the present game). 
Relying on \eqref{2311212} we will now prove that \eqref{afasfqwreqrqwrf} holds.

Observe:

\begin{enumerate}[(i)]

\item $\tau_2^* \geq \tau^{(a,b)^c} = \tilde \tau_2^*$ a.s. for the present starting value $X_0=x$. Note that we do not necessarily have equality here since $\tau_2^*$ relies on randomization. 

\item We know that if $|a|,b<\infty$ then $a,b \in \partial D_2^{*}\cap \overline{{B^{f \lesssim h \lesssim g}} \cup {B^{f\leq g<h}} \cup {B^{h<f\leq g}}}$. We also know that $f\leq g$ on $\overline{{B^{f \lesssim h \lesssim g}} \cup {B^{f\leq g<h}} \cup {B^{h<f\leq g}}}$. 
Hence,
\begin{align*}
e^{-r \tau^{(a,b)^c}} f(X_{\tau^{(a,b)^c}}) \leq  e^{-r \tau^{(a,b)^c}} g(X_{\tau^{(a,b)^c}}).
\end{align*}

\item Note that $\mathbb{I}{\{\tau_2^*=\tau^{(a,b)^c}<\infty\}}(\omega)= 1 \Rightarrow X_{ \tau_2^* }(\omega) \in \overline{{B^{f \lesssim h \lesssim g}} \cup {B^{f\leq g<h}} \cup {B^{h<f\leq g}}}\backslash \mbox{int}(B^{f\leq g<h})$ a.s., since Player $2$ can only stop on $\mbox{int}(B^{f\leq g<h})$ with randomization. 

Note that $h\leq g$ on $\overline{{B^{f \lesssim h \lesssim g}} \cup {B^{f\leq g<h}} \cup {B^{h<f\leq g}}}\backslash \mbox{int}(B^{f\leq g<h})$.
Hence, 
\begin{align*}
e^{-r {\tau_2^*}}h(X_{\tau_2^*})\mathbb{I}{\{\tau^{(a,b)^c}=\tau_2^*\}}\mathbb{I}{\{\tau_1=\tau^{(a,b)^c}\}}
\leq 
e^{-r \tau^{(a,b)^c}}g(X_{\tau^{(a,b)^c}})\mathbb{I}{\{\tau^{(a,b)^c}=\tau_2^*\}}\mathbb{I}{\{\tau_1=\tau^{(a,b)^c}\}}.
\end{align*}
\item $J(y;\tau_1,\tau_2^*) \leq g(y)+\epsilon/2$ for $y=a,b$, if $|a|,b<\infty$. Using $a,b \in \partial D_2^{*}\cap \overline{{B^{f \lesssim h \lesssim g}} \cup {B^{f\leq g<h}} \cup {B^{h<f\leq g}}}$ we find that this has been proved in the items above, cf. e.g., \eqref{adqwefas}.
\end{enumerate}
Using these observations and the strong Markov property we find 
\begin{align*}
&  \sup_{\tau_1 \in \mathbb{T}_1}  J(x; \tau_1,\tau_2^*)\\
& = \sup_{\tau_1 \in \mathbb{T}_1}  \E_x\bigg(e^{-r(\tau_1 \wedge  \tau_2^*)}\bigg(f(X_{\tau_1})\mathbb{I}{\{\tau_1<\tau_2^*\}} +   g(X_{\tau_2^*})\mathbb{I}{\{\tau_1>\tau_2^*\}} + h(X_{\tau_2^*})\mathbb{I}{\{\tau_1=\tau_2^*\}}  \bigg) \bigg)\\
& \leq \sup_{\tau_1 \in \mathbb{T}_1}  \E_x\bigg(e^{-r(\tau_1 \wedge  \tau^{(a,b)^c})}\bigg( f(X_{\tau_1})\mathbb{I}{\{\tau_1<\tau^{(a,b)^c}\}}+  g(X_{\tau^{(a,b)^c}})\mathbb{I}{\{\tau_1=\tau^{(a,b)^c}\}}   \\
  & \qquad\qquad\qquad\qquad\qquad\qquad\enskip + J(X_{\tau^{(a,b)^c}};\tau_1,\tau_2^*)\mathbb{I}{\{\tau_1>\tau^{(a,b)^c}\}} \bigg) \bigg)\\
& \leq \sup_{\tau_1 \in \mathbb{T}_1}  \E_x\bigg(e^{-r(\tau_1 \wedge  \tau^{(a,b)^c})}\bigg( f(X_{\tau_1})\mathbb{I}{\{\tau_1<\tau^{(a,b)^c}\}} + (g(X_{\tau^{(a,b)^c}})+\epsilon/2)\mathbb{I}{\{\tau_1\geq\tau^{(a,b)^c}\}}  \bigg) \bigg)\\
& \leq \sup_{\tau_1 \in \mathbb{T}_1}  \E_x\bigg(e^{-r(\tau_1 \wedge  \tau^{(a,b)^c})}\bigg( f(X_{\tau_1})\mathbb{I}{\{\tau_1<\tau^{(a,b)^c}\}}+  g(X_{\tau^{(a,b)^c}})\mathbb{I}{\{\tau_1\geq \tau^{(a,b)^c}\}} \bigg) \bigg) +  \epsilon/2.
\end{align*}
The inequalities above and \eqref{2311212} imply that 
\begin{align} \label{afasfqwreqrqwrf2}
\sup_{\tau_1 \in \mathbb{T}_1} J(x; \tau_1,\tau_2^*)-\epsilon/2 \leq \tilde J(x; \tilde \tau_1^*,\tilde \tau_2^*).
\end{align}

Observe that the following holds for the present $X_0=x$:

\begin{enumerate}[(i)]

\item  $\tilde \tau_i^* \leq \tau_i^*$ and $\tilde \tau_2^* = \tau^{(a,b)^c}$ a.s.

\item Using e.g., \eqref{adqwefas3} we obtain $J(X_{\tilde \tau_2^* };\tau_1^*,\tau_2^*)  \geq g(X_{\tilde \tau_2^*}) - \epsilon/2$ a.s. 

\item It also holds that $J(X_{\tilde \tau_1^* };\tau_1^*,\tau_2^*) \geq f(X_{ \tilde \tau_1^* }) - \epsilon/2$ a.s. Indeed, this inequality can obtained using the equilibrium stopping time for Player $1$ (cf. Construction \ref{construction2}), in analogy with how 
$J(y;\tau_1,\tau_2^*) \leq g(y)+\epsilon/2$ for $y=a,b$ was obtained above.

\end{enumerate}

Using this, we obtain for the present $X_0=x$ that 
\begin{align*}
& J(x; \tau_1^*, \tau_2^*)\\
& = \E_x\bigg(e^{-r(\tilde \tau_1^* \wedge \tilde \tau_2^*)}   J(X_{\tilde \tau_1^* \wedge \tilde \tau_2^*}; \tau_1^*, \tau_2^*)   \bigg)\\
& = \E_x\bigg(e^{-r(\tilde \tau_1^* \wedge \tilde \tau_2^*)}    \bigg( 
J(X_{\tilde \tau_1^*}; \tau_1^*,\tau_2^*)\mathbb{I}{\{\tilde\tau_1^* \leq \tilde \tau_2^*\}}+
J(X_{\tilde \tau_2^*}; \tau_1^*,\tau_2^*)\mathbb{I}{\{\tilde\tau_1^*> \tilde \tau_2^*\}}
\bigg) \bigg)\\
&  \geq  \E_x\bigg(
e^{-r(\tilde\tau_1^* \wedge \tilde \tau_2^*)}\bigg(
f(X_{\tilde\tau_1^*})\mathbb{I}{\{\tilde\tau_1^*\leq \tilde \tau_2^*\}}   + 
g(X_{\tilde\tau_2^*})\mathbb{I}{\{\tilde\tau_1^* > \tilde \tau_2^*\}} 
\bigg) - \epsilon/2\\
&  = \tilde J(x; \tilde \tau_1^*, \tilde \tau_2^*)- \epsilon/2,
\end{align*}
where the last equality also relies on \eqref{2311212}. Now use \eqref{afasfqwreqrqwrf2} and the inequality  above to see that \eqref{afasfqwreqrqwrf} holds.

\item  The remaining points are $x \in {B^{g<f<h}}$. The proposed equilibrium corresponds to Player $1$ stopping immediately and Player $2$ randomizing.  
The corresponding value is $J(x;\tau^*_1,\tau^*_2)=f(x)$. 
Hence, if we can show, for any fixed $\epsilon>0$, that
$$J(x;\tau_1,\tau^*_2) \leq f(x)+\epsilon,$$
for each $\tau_1 \in \mathbb{T}_1$, then the first part of \eqref{def:SE} holds, and we are done; but this statement follows from arguments similar to those for $x \in  {B^{f\leq g<h}} \cap D_1^* \cap D_2^*$, cf. \eqref{234rasf}, when noting that ${B^{g<f<h}} = {B^{g<f<h}} \cap D_1^* \cap D_2^*$ and that $g<f$ on this set.
 \end{enumerate}
 
\hfill \qedsymbol{}\\

\noindent \textbf{Proof of  Corollary \ref{Equilibrium-uniqueness}.} 
The first result can be proved by following the same arguments as those in the proof of Theorem \ref{thm:construction}. 
The second result corresponds to a basic result for general zero-sum games; cf. Section \ref{sec:intro}.  
\hfill \qedsymbol{}

\section{Proofs for Section \ref{sec:howtofind}} \label{proof-app-D}

 The proof of Theorem \ref{simplified-construction} relies on the following result, whose first part has this interpretation: On any interval $B$ where the payoff function of the maximizer dominates the payoff function of the minimizer, i.e., $f>g$, it holds that the maximizer can randomize stopping so that the minimizer wants to stop immediately, assuming that the minimizer stops alone at the latest when leaving $B$ (or that both stop when leaving $B$ and $g=h$ on $\partial B$). The interpretation of the second result is analogous.

\begin{proposition}\label{useful-prop} Suppose $B\subseteq {\cal I}$ is an interval with $g(x)<f(x)$ for all $x\in B$.
\begin{enumerate}[(i)]

\item  Let  
\begin{align*}
A_t:= \int_0^t \frac{1}{f(X_s)-g(X_s)}\mathbb{I}_{\{X_s\in B\}}dC_s^{-}(g), \enskip t \geq 0. 
\end{align*}
Then $A_t<\infty, t < \tau^{\partial B}$ a.s. (while it may occur that $A_{\tau^{\partial B}}=\infty$) and with $\tau_{A}:=\inf\{t\geq 0: A_t\geq E_1\}$ it holds, for all $x \in B$ and $\tau_2 \in \mathbb{T}_2$, that
\begin{align*}
& g (x)\\
& \leq \E_x\left(e^{-r(\tau_A \wedge \tau_2)} \left(
f(X_{\tau_A})\mathbb{I}{\{\tau_A<\tau_2\}} + g(X_{\tau_2})\mathbb{I}{\{\tau_A>\tau_2\}} + h(X_{\tau_2})\mathbb{I}{\{\tau_2=\tau_{A}\}}\right)\mathbb{I}{\{\tau_2 \wedge \tau_A < \tau^{\partial B}\}}\right)\\
& + \E_x\left(e^{-r\tau^{\partial B}}  g(X_{\tau^{\partial B}}) \mathbb{I}{\{\tau_A \wedge \tau_2 \geq  \tau^{\partial B}\}}\right).
\end{align*}
Moreover, if $g$ is twice continuously differentiable except on a countable set of separated points $x_j \in {\cal I},j \in J$ (where $J$ is an index set), then
\begin{align*}
A_t = \int_0^t\left(\frac{(\mathbb{L}_X-r)g(X_s)}{g(X_s)-f(X_s)}\right)_+\mathbb{I}_{\{X_s \in B, X_s \neq x_j,\forall j\}}ds + \frac{1}{2}\sum_{j} \mathbb{I}_{\{x_j \in B\}}\left(\frac{g'(x_j+)-g'(x_j-)}{{g(x_j)-f(x_j)}}\right)_+l_t^{x_j}, \enskip t \geq 0.
\end{align*}

\item Let 
\begin{align}\label{vagew}
A_t:= \int_0^t \frac{1}{f(X_s)-g(X_s)}\mathbb{I}_{\{X_s\in B\}}dC_s^{+}(f), \enskip t \geq 0. 
\end{align}
Then $A_t<\infty, t < \tau^{\partial B}$ a.s. (while it may occur that $A_{\tau^{\partial B}}=\infty$) and with $\tau_{A}:=\inf\{t\geq 0: A_t\geq E_2\}$ it holds, for all $x \in B$ and $\tau_1 \in \mathbb{T}_1$, that
\begin{align*}%\label{nqrhklq35}
\begin{split}
& f (x)\\
& \geq \E_x\left(e^{-r(\tau_1 \wedge \tau_{A})} \left(
f(X_{\tau_1})\mathbb{I}{\{\tau_1<\tau_{A}\}} + g(X_{\tau_{A}})\mathbb{I}{\{\tau_1>\tau_{A}\}} + h(X_{\tau_1})\mathbb{I}{\{\tau_1=\tau_{A}\}}
\right)\mathbb{I}{\{\tau_1 \wedge \tau_{A} < \tau^{ \partial B }\}}\right)\\
& + \E_x\left(e^{-r\tau^{\partial B}}  f(X_{\tau^{\partial B}}) \mathbb{I}{\{\tau_1 \wedge \tau_{A} \geq  \tau^{\partial B}\}}\right).
\end{split}
\end{align*}
Moreover, if $f$ is twice continuously differentiable except on a countable set of separated points $x_i \in {\cal I}, i \in I$ (where $I$ is an index set) then
\begin{align*}
A_t =  \int_0^t\left(\frac{(\mathbb{L}_X-r)f(X_s)}{f(X_s)-g(X_s)}\right)_+\mathbb{I}_{\{X_s \in B, X_s \neq x_i,\forall i\}}ds + 
\frac{1}{2}\sum_{i} \mathbb{I}_{\{x_i \in B\}} \left(\frac{f'(x_i+)-f'(x_i-)}{{f(x_i)-g(x_i)}}\right)_+l_t^{x_i}, \enskip t \geq 0.
\end{align*}

\end{enumerate}

\end{proposition}

\begin{proof} 
We present a proof of result (ii). One can prove result (i) analogously.  

 The first assertion holds since $C_t^+(f), t \geq 0$ takes values in $[0,\infty)$, cf. Section \ref{sec:assum}, and $f>g$ on $B$. The last assertion follows from the It\^o-Meyer formula and the occupation density formula applied to the definition of $C_t^+(f)$ as already carried out in Section \ref{sec:assum}. 
 
Let us prove the second assertion:
Using arguments analogous to those in \cite[Proposition 6]{bodnariu2022local} we obtain, for all $\tau \in \mathbb{T}_1$, that a.s.
	\begin{align*}%\label{eq:key_conditioning} 
		\E_x\left(e^{-r(\tau \wedge \tau_{A})} \left(
		f(X_{\tau})\mathbb{I}{\{\tau<\tau_{A}\}} + g(X_{\tau_{A}})\mathbb{I}{\{\tau>\tau_{A}\}}\right) |{\cal F}^X_\infty\right)= 
  \E_x\left(e^{-r\tau -A_\tau}f(X_\tau) +\int_0^\tau e^{-rs -  A_s}g(X_s)dA_s|{\cal F}^X_\infty\right).
	\end{align*}
Using \eqref{2q4124} we also have 
	\begin{align*}
		d(e^{-A_t}e^{-rt}f(X_t))&=e^{-rt}f(X_t)de^{-A_t}+e^{-A_t}d(e^{-rt}f(X_t))\\
		&=-e^{-rt}f(X_t)e^{-A_t}dA_t+e^{-A_t}\left(e^{-rt}dC_t^{+}(f)-e^{-rt}dC_t^{-}(f)+dM_t\right),
	\end{align*}
where $(M_t)$ is a local martingale. From this and \eqref{vagew} we obtain
\begin{align*}%\label{3qrmwet}
 \begin{split}
		&\;d\left(e^{-rt -A_t}f(X_t)+\int_0^t e^{-rs -A_s}g(X_s)dA_s\right)\\
		&=-e^{-rt-A_t}f(X_t)dA_t+e^{-A_t}\left(e^{-rt}dC_t^{+}(f)-e^{-rt}dC_t^{-}(f)+dM_t\right)+e^{-rt-A_t}g(X_t)dA_t \\
		&=-e^{-rt -A_t}f(X_t)dA_t+e^{-A_t}\left(e^{-rt}(f(X_t)-g(X_t))dA_t-e^{-rt}dC_t^{-}(f)+dM_t\right)+e^{-rt-A_t}g(X_t)dA_t\\
		&=e^{-A_t}dM_t-e^{-rt -A_t}dC_t^{-}(f).
  \end{split}
	\end{align*}
From now on let $\tau:=\tau_1 \wedge \tau^{\partial B}$.  Observe that by definition we directly have
\begin{align*}%\label{q34fgdad}
\begin{split}
&\E_x\left(e^{-r(\tau \wedge \tau_{A})} \left(
		f(X_{\tau})\mathbb{I}{\{\tau < \tau_{A}\}} + g(X_{\tau_{A}})\mathbb{I}{\{\tau>\tau_{A}\}}\right) \right)\\
  & = \E_x\left(e^{-r(\tau_1 \wedge \tau_{A})} \left(
		f(X_{\tau_1})\mathbb{I}{\{\tau_1< \tau_{A}\}} + g(X_{\tau_{A}})\mathbb{I}{\{\tau_1>\tau_{A}\}}
            \right)\mathbb{I}{\{\tau_1 \wedge \tau_{A} < \tau^{\partial B} \}}\right)\\
    & \enskip  + \E_x\left(e^{-r\tau^{\partial B}} 
		f(X_{\tau^{\partial B}})\mathbb{I}{\{\tau^{\partial B}< \tau_{A}\}}
            \mathbb{I}{\{\tau_1 \wedge \tau_{A} \geq \tau^{\partial B} \}}\right).
\end{split}
\end{align*}
Let us now show that 
\begin{align}\label{q3qgwr2t}
\E_x\left(e^{-r\tau^{\partial B}}  
		f(X_{\tau^{\partial B}})\mathbb{I}{\{\tau^{\partial B}< \tau_{A}\}}
            \mathbb{I}{\{\tau_1 \wedge \tau_{A} \geq \tau^{\partial B} \}}\right)
            = 
\E_x\left(e^{-r\tau^{\partial B}} 
		f(X_{\tau^{\partial B}})
             \mathbb{I}{\{\tau_1 \wedge \tau_{A} \geq \tau^{\partial B} \}}\right).
\end{align}
Note here that we have the behavior on $\{\tau^{\partial B}=\infty\}$ consistent with assumption \eqref{at-infty-assum}. Using this and 
\begin{align*}
\mathbb{I}{\{\tau_1 \wedge \tau_{A} \geq \tau^{\partial B} \}}
& = \mathbb{I}{\{\tau_1 \wedge \tau_{A} \geq \tau^{\partial B} \}}\mathbb{I}{\{\tau^{\partial B}< \tau_A \}}+
\mathbb{I}{\{\tau_1 \wedge \tau_{A} \geq \tau^{\partial B} \}}\mathbb{I}{\{\tau^{\partial B} \geq \tau_A \}}\\
& = \mathbb{I}{\{\tau_1 \wedge \tau_{A} \geq \tau^{\partial B} \}}\mathbb{I}{\{\tau^{\partial B}< \tau_A \}}+
\mathbb{I}{\{\tau_1 \wedge \tau_{A} \geq \tau^{\partial B} \}}\mathbb{I}{\{\tau^{\partial B} = \tau_A \}},
\end{align*}
we see that if it holds that
\begin{align}\label{1325rgw}
            \mathbb{I}{\{\tau_{A}= \tau^{\partial B}< \infty \}}= 0, 
\end{align}
a.s., then \eqref{q3qgwr2t} follows. To see that \eqref{1325rgw} holds, use that $(A_t)$ is a continuous process, the definition of $\tau_A$ and that $E_i$ has a continuous distribution (cf. \cite[Lemma 4.4]{christensen2024existence}). Hence, \eqref{q3qgwr2t} holds.

It can also be verified that $\{\tau_1 = \tau_A < \tau^{\partial B}\}$ is a null set. To see this use that $(A_t)$ is finite on $\{t<\tau^{\partial B}\}$, that $E_i$ are independent and have a continuous distribution and the definition of $\tau_{A}$ (again, cf. \cite[Lemma 4.4]{christensen2024existence}). 
This implies that we have  the first equality in \eqref{31gewsd} below.   

Putting also the other pieces above together we find, for any $\tau:=\tau_1 \wedge \tau^{\partial B}$ 
that is dominated by an exit time from a bounded interval, that  
\begin{align}\label{31gewsd}
	\begin{split}
		& \E_x\left(e^{-r(\tau_1 \wedge \tau_{A})} \left(
		f(X_{\tau_1})\mathbb{I}{\{\tau_1<\tau_{A}\}} + g(X_{\tau_{A}})\mathbb{I}{\{\tau_1>\tau_{A}\}} + h(X_{\tau_1})\mathbb{I}{\{\tau_1=\tau_{A}\}}
		\right)\mathbb{I}{\{\tau_1 \wedge \tau_{A} < \tau^{\partial B}\}}\right)\\
		& + \E_x\left(e^{-r\tau^{\partial B}}  f(X_{\tau^{\partial B}}) \mathbb{I}{\{\tau_1 \wedge \tau_{A} \geq  \tau^{\partial B}\}}\right)\\
  & = \E_x\left(e^{-r(\tau_1 \wedge \tau_{A})} \left(
		f(X_{\tau_1})\mathbb{I}{\{\tau_1<\tau_{A}\}} + g(X_{\tau_{A}})\mathbb{I}{\{\tau_1>\tau_{A}\}}	\right)\mathbb{I}{\{\tau_1 \wedge \tau_{A} < \tau^{\partial B}\}}\right)\\
		& + \E_x\left(e^{-r\tau^{\partial B}}  f(X_{\tau^{\partial B}}) \mathbb{I}{\{\tau_1 \wedge \tau_{A} \geq  \tau^{\partial B}\}}\right)\\
		 &  =\E_x\left(e^{-r(\tau \wedge \tau_{A})} \left(
		f(X_{\tau})\mathbb{I}{\{\tau < \tau_{A}\}} + g(X_{\tau_{A}})\mathbb{I}{\{\tau>\tau_{A}\}}\right) \right)\\
		& = f(x)+\E_x\left(\int_0^\tau  (e^{-A_s}dM_s-e^{-rs -A_s}dC_s^{-}(f))\right)\\
		&  \leq  f(x)+\E_x\left(\int_0^\tau e^{-A_s}dM_s\right) =  f(x).
\end{split}
\end{align}
The general statement now follows from \eqref{at-infty-assum}, \eqref{assum:integrability} and dominated convergence. 
\end{proof}
We also need the following result which says that the equilibrium stopping time randomization intensity processes, i.e., $(A_t^{(i)})_{t\geq 0}$, do not blow up before $(X_t)_{t \geq 0}$ reaches $\partial {B^{g\leq h \leq f}}$. It may, however, be the case that they blow at this point, but we also note that both equilibrium stopping times prescribe immediate stopping on ${B^{g\leq h \leq f}}$. 

\begin{proposition}\label{prop:simul-stopping2X} Let $\Psi_t^{(i)},A_t^{(i)},i=1,2$ be as in Theorem \ref{simplified-construction}.
\begin{enumerate}[(i)]

\item It holds that $A_t^{(1)},A_t^{(2)}<\infty$ for $t<\tau^{\partial {B^{g\leq h \leq f}}}$, $\mathbb{P}_x$-a.s. 

\item Consider an arbitrary fixed $\tau_1 \in \mathbb{T}_1$. If $\tau_1(\omega)=\tau_{\Psi^{(2)}}(\omega)<\infty$ then $X_{\tau_1}(\omega)\in \overline{D_2^* \backslash {B^{g<f<h}}}$, $\mathbb{P}_x$-a.s. Similarly, consider an arbitrary fixed $\tau_2 \in \mathbb{T}_2$. If $\tau_2(\omega)=\tau_{\Psi^{(1)}}(\omega)<\infty$ then $X_{\tau_2}(\omega)\in \overline{D_1^* \backslash {B^{h<g<f}}}$, $\mathbb{P}_x$-a.s. 
\end{enumerate}

\end{proposition}
\begin{proof} 
(i) This follows from Proposition \ref{useful-prop}. In particular, to see this for $A_t^{(1)}$, we note that $f>g$ on ${B^{h<g<f}}$ and that the assumption ${B^{f\leq g<h}}\cup {B^{h<f\leq g}}=\emptyset$ implies that 
$x \in \partial ({B^{h<g<f}}) \Rightarrow x \in \partial ({B^{g\leq h \leq f}})$. The proof for $A_t^{(2)}$ is analogous. 
(ii) This holds by Proposition \ref{prop:simul-stopping}(ii). 
\end{proof}

\noindent\textbf{Proof of Theorem \ref{simplified-construction}.} The last part of the result follows directly from the first part, using Proposition \ref{useful-prop}.  

The remainder of this proof follows the same structure as the proof of Theorem \ref{thm:construction}. 
Indeed to show that the first inequality in \eqref{def:SE} holds with $\epsilon=0$ for $x\in {B^{g\leq h \leq f}} \cup {B^{h<g<f}}$, we use the same arguments as in the proof of Theorem \ref{thm:construction}, relying also on Proposition \ref{prop:simul-stopping2X}. This also holds for $x\in {B^{f \lesssim h \lesssim g}}\cap D_2^*$. Using the assumption ${B^{f\leq g<h}} \cup {B^{h<f\leq g}}=\emptyset$ we can moreover adjust the arguments that lead to \eqref{afasfqwreqrqwrf} to see that this inequality holds with $\epsilon=0$, for $x\in {B^{f \lesssim h \lesssim g}}\cap (D_2^*)^c$.   
Hence, the first inequality in \eqref{def:SE} holds with $\epsilon=0$ for $x\in {B^{f \lesssim h \lesssim g}}$, and it only remains to show this for ${B^{g<f<h}}$: 

Set $\tau_{A}:=\inf\{t\geq 0: A_t^{(2)}\geq E_1\}$. Using the strong Markov property, 
that $x \in \partial {B^{g<f<h}} \Rightarrow x \in \partial {B^{g\leq h \leq f}}$ (since ${B^{f\leq g<h}}=\emptyset$) implying  $J(X_{\tau^{\partial {B^{g<f<h}}}}; \tau_1,\tau_2^*)\leq h(X_{\tau^{\partial {B^{g<f<h}}}})=f(X_{\tau^{\partial {B^{g<f<h}}}})$ a.s.  
and Proposition \ref{useful-prop} (for the last inequality), we find, for any $x\in {B^{g<f<h}}$ and $\tau_1 \in \mathbb{T}_1$, that 
\begin{align*}
&  J(x; \tau_1,\tau_2^*)\\
& =  \E_x\bigg(e^{-r(\tau_1 \wedge  \tau_2^*)}\bigg(f(X_{\tau_1})\mathbb{I}{\{\tau_1<\tau_2^*\}} + g(X_{\tau_2^*})\mathbb{I}{\{\tau_1>\tau_2^*\}}  + h(X_{\tau_2^*})\mathbb{I}{\{\tau_1=\tau_2^*\}} \bigg) \bigg)\\
& =  \E_x\bigg(e^{-r(\tau_1 \wedge  \tau_2^*)}\bigg(f(X_{\tau_1})\mathbb{I}{\{\tau_1<\tau_2^*\}} +   g(X_{\tau_2^*})\mathbb{I}{\{\tau_1>\tau_2^*\}} + h(X_{\tau_2^*})\mathbb{I}{\{\tau_1=\tau_2^*\}} \bigg) \mathbb{I}{\{\tau_1 \wedge \tau_2^* <\tau^{\partial {B^{g<f<h}}}\}} \bigg)\\
& \enskip +  \E_x\bigg(e^{-r\tau^{\partial {B^{g<f<h}}}}    J(X_{\tau^{\partial {B^{g<f<h}}}}; \tau_1,\tau_2^*)      \mathbb{I}{\{\tau_1 \wedge \tau_2^* \geq \tau^{\partial {B^{g<f<h}}}\}}\bigg)\\
& \leq \E_x\bigg(e^{-r(\tau_1 \wedge  \tau_2^*)}\bigg(f(X_{\tau_1})\mathbb{I}{\{\tau_1<\tau_2^*\}} +   g(X_{\tau_2^*})\mathbb{I}{\{\tau_1>\tau_2^*\}} + h(X_{\tau_2^*})\mathbb{I}{\{\tau_1=\tau_2^*\}}\bigg) \mathbb{I}{\{\tau_1 \wedge \tau_2^* < \tau^{\partial {B^{g<f<h}}}\}} \bigg)\\
& \enskip +  \E_x\bigg(e^{-r\tau^{\partial {B^{g<f<h}}}  }    f(X_{\tau^{\partial {B^{g<f<h}}}}) \mathbb{I}{\{\tau_1 \wedge \tau_2^* \geq  \tau^{\partial {B^{g<f<h}}}\}}\bigg)\\
& \leq \E_x\bigg(e^{-r(\tau_1 \wedge  {\tau_{  A}})}\bigg(f(X_{\tau_1})\mathbb{I}{\{\tau_1<{\tau_{  A}}\}} +   g(X_{{\tau_{  A}}})\mathbb{I}{\{\tau_1>{\tau_{  A}}\}} + h(X_{{\tau_{  A}}})\mathbb{I}{\{\tau_1={\tau_{  A}}\}}\bigg) \mathbb{I}{\{\tau_1 \wedge {\tau_{  A}} < \tau^{\partial {B^{g<f<h}}}\}} \bigg)\\
& \enskip +  \E_x\bigg(e^{-r\tau^{\partial {B^{g<f<h}}}  }    f(X_{\tau^{\partial {B^{g<f<h}}}}) \mathbb{I}{\{\tau_1 \wedge {\tau_{A}} \geq  \tau^{\partial {B^{g<f<h}}}\}}\bigg)\\
& \leq f(x) = J(x; \tau_1^*,\tau_2^*),
\end{align*}
where the last equality holds since Player $1$ stops without randomization on ${B^{g<f<h}}$ (which is an open set). Hence, the first inequality in \eqref{def:SE} holds with $\epsilon=0$ for $x\in {B^{g<f<h}}$, i.e., the equilibrium condition is satisfied for Player $1$.  In view of Construction \ref{construction2} it is clear that the corresponding result for Player $2$ also holds (cf. the remarks made in the beginning of the proof of Theorem \ref{thm:construction}). \hfill \qedsymbol{}

\section{Proofs for Section \ref{sec:NE-condi}} \label{proof-app-E}

\noindent \textbf{Proof of Theorem \ref{thm:NE-conditions2}.}
Suppose condition (i) holds. Suppose, in order to obtain a contradiction, that a global Markovian pure Nash equilibrium exists; denote its value by $V$. Consider an arbitrary point $x\in {B^{g<f<h}} \cup {B^{f\leq g<h}}$.  
Observe that both players stopping at $x$, giving the payoff $h(x)$, cannot happen in equilibrium since Player $2$ would benefit from deviating (since $f(x)<h(x)$). 
Observe that Player $2$ stopping at $x$ and Player $1$ not stopping at $x$, giving the payoff $g(x)$, cannot happen in equilibrium since Player $1$ would benefit from deviating (since $h(x)>g(x)$).

These observations imply that for the global Markovian pure Nash equilibrium (which we have supposed exists), it holds that Player $2$ does not stop on ${B^{g<f<h}} \cup {B^{f\leq g<h}}$. 
Now, given a Markovian pure stopping time for Player $2$ that never stops on ${B^{g<f<h}} \cup {B^{f\leq g<h}}$ it holds that Player $1$ can always obtain a reward corresponding to the left hand side in \eqref{asdasda}, i.e., the equilibrium value $V$ has to satisfy
\begin{align*}
V(x_0) \geq \sup_{\tau < \tau^{({B^{g<f<h}} \cup {B^{f\leq g<h}})^c}}\mathbb{E}_{x_0}(e^{-r\tau}f(X_\tau)) > f (x_0) \vee g(x_0)
\end{align*}
(where the last inequality also relies on \eqref{asdasda}). This contradicts the result that the equilibrium value $V$ is found between $f$ and $g$ (cf. Theorem \ref{thm:construction} and Proposition \ref{asso-game:characterization}). We have thus reached a contradiction. Hence, no global Markovian pure Nash equilibrium exists under condition (i). The result for condition (ii) can be proved analogously.
\hfill \qedsymbol{}

\bibliographystyle{abbrv}
\bibliography{Bibl}

\end{document}